\documentclass[a4paper,10pt]{article}
\title{A higher homotopic extension of persistent (co)homology}
\author{Estanislao Herscovich
\footnote{Departamento de Matem\'atica, FCEyN, Universidad de Buenos Aires, Argentina. 
The author is also a research member of CONICET (Argentina). 
On leave of absence from Institut Joseph Fourier, Universit\'e Grenoble I, France. 
The author would also like to thank the Alexander von Humboldt Foundation and the Bielefeld University, for their support during part of this work. This work was also partially supported by UBACYT 20020130200169BA, UBACYT 20020130100533BA, PIP-CONICET 2012-2014 11220110100870, 
MathAmSud-GR2HOPF, PICT 2011-1510 and PICT 2012-1186.}}
\date{}

\usepackage{amsmath,amsthm,amsfonts,amssymb,stmaryrd,mathrsfs,enumitem}
\usepackage{latexsym}
\usepackage{fullpage}
\usepackage{a4,color,palatino,fancyhdr}
\usepackage{graphicx}%poner [dvips] para gráficos
\usepackage{float}  %Este packete permite controlar la flotacion de figuras al poner [H] o [h] como opcion de \begin{figure}
\usepackage[small, bf, margin=90pt, tableposition=bottom]{caption}
\RequirePackage[T1]{fontenc}
\usepackage[breaklinks, naturalnames]{hyperref}
\usepackage{amsrefs}

\input{xy}
\xyoption{all}
\xyoption{poly}
\usepackage[all]{xy}

\newtheorem{theorem}{Theorem}[section]
\newtheorem{proposition}[theorem]{Proposition}

\newtheorem{remark}[theorem]{Remark}

\newtheorem{fact}[theorem]{Fact}

\numberwithin{equation}{section}

\def\place{{-}}

\def\Ker{\mathop{\rm Ker}\nolimits}

\newcommand\NN{{\mathbb{N}}}
\newcommand\ZZ{{\mathbb{Z}}}
\newcommand\RR{{\mathbb{R}}}

\def\place{{-}}

\begin{document}

\maketitle
                                                     
\hrulefill
%%%%%%%%%%%%%%%%%%%%%%%%%%%%%%%%%%%%%%%%%%%%%%%%%%%%%%%%%%%%%%%%%%%%%%%%%%%%%%%%%%%%%%%%%%%%%%%%%%%%%%%%%%%%%%%%%%%%%%%%%%%%%%%%%%%%%%%%%%%%%%%%
\begin{abstract}
\noindent Our objective in this article is to show a possibly interesting structure of homotopic nature appearing in persistent (co)homology. 
Assuming that the filtration of the (say) simplicial set embedded in $\RR^{n}$ induces a multiplicative filtration (which would not be a so harsh hypothesis in our setting) 
on the dg algebra given by the complex of simplicial cochains, 
we may use a result by T. Kadeishvili to get a unique (up to noncanonical equivalence) $A_{\infty}$-algebra structure on the complete persistent cohomology of the filtered simplicial (or topological) set. 
We then provide a construction of a (pseudo)metric on the set of all (generalized) barcodes (that is, of all cohomological degrees) 
enriched with the $A_{\infty}$-algebra structure stated before, refining the usual bottleneck metric, and which is also independent of the particular $A_{\infty}$-algebra structure chosen 
(among those equivalent to each other).
We think that this distance might deserve some attention for topological data analysis, for it in particular can recognize different linking or foldings patterns, as in the Borromean rings. 
As an aside, we give a simple proof of a result relating the barcode structure between persistent homology and cohomology. 
This result was observed in \cite{dSMVJ11bis} under some restricted assumptions, which we do not suppose. 
\end{abstract}

\textbf{Mathematics subject classification 2010:} 16E45, 16W70, 18G55, 55U10, 68U05.

\textbf{Keywords:} persistent homology, dg algebras, $A_{\infty}$-algebras, bottleneck metric. 

\hrulefill
%%%%%%%%%%%%%%%%%%%%%%%%%%%%%%%%%%%%%%%%%%%%%%%%%%%%%%%%%%%%%%%%%%%%%%%%%%%%%%%%%%%%%%%%%%%%%%%%%%%%%%%%%%%%%%%%%%%%%%%%%%%%%%%%%%%%%%%%%%%%%%%%
\section{Introduction}

The aim of this article is to extend some of the constructions appearing in persistent (co)homology to consider further algebraic structures. 
More precisely, under the assumption that an increasing filtration of a simplicial set (or a topological space) embedded in $\RR^{n}$ induces a 
decreasing multiplicative filtration of the dg algebra given by the complex of simplicial (or singular) cochains,  
we consider the $A_{\infty}$-algebra structure on the cohomology of the latter dg algebra, which is precisely the (complete) persistent cohomology of the filtered simplicial set (or the topological space). 
We have to remark here that, unlike the situation where the indexing set for the filtration is (a subset of) $\ZZ$ and the algebraic properties of the latter are useful to simplify 
the computations by regarding some associated algebraic structures, in the context of topological data analysis these conditions are only hardly satisfied, 
since for the latter the filtrations are more or less imposed, as they are constructed from the metric of the embedded simplicial set 
(or the topological space).   
For this reason we prefer to relax the hypotheses on the product of the set of indices $P$ (and even on the set itself), 
which also led us to consider more naturally \emph{exact couples systems} (rather than classical exact couples), which were introduced by B. Matschke in his preprint \cite{Ma13}. 
For instance we may take as the product in $P$ the one given by taking the infimum, which is the coarsest possible in case $P$ is totally ordered and the members of the filtration 
are dg subalgebras of the dg algebra under study. 
These hypotheses are always satisfied in the case of the filtration of the dg algebra of normalized cochains on a simplicial set provided with the induced filtration from any filtration 
of the simplicial set. 

The main contribution of this article is the construction of a distance on the collection of all generalized barcodes (\textit{i.e.} including all the barcodes of all cohomological degrees) 
together with the $A_{\infty}$-algebra on the complete persistent cohomology described before, 
which refines the usual bottleneck metric and which is invariant under quasi-isomorphisms of $A_{\infty}$-algebras. 
Even though the $A_{\infty}$-algebra structure considered is only a shadow of the complete structure lurking behind, 
for it does not take (completely) into account the structure coming from the action of the category associated to the poset of indices 
(see the last paragraph of Subsection \ref{subsection:gen}),
we believe that this metric may still be of interest for analysing data structures. 
In particular, since the $A_{\infty}$-algebra structure of the cohomology of a dg algebra extends in a systematic manner the Massey products (see \cite{LPWZ09}, Thm. 3.1 and Cor. A.5), 
we believe that it may be used for recognizing linking or foldings patterns, as in the typical case of Borromean rings in contrast to unlinked rings (\textit{cf.} \cite{MSS97}), etc. 

The possible algebraic (resp., coalgebraic) structures on persistent cohomology (resp., homology) have not deserved much attention up to the present time in the literature as far we know. 
We want to remark that in the preprint article \cite{BM14} by F. Belch\'{\i} and A. Murillo, the authors seem to proceed in our opinion in a different manner as ours (regardless of the fact they consider $A_{\infty}$-coalgebras, as pointed out in the last paragraph of this section), since they produce barcodes by means of the higher structure maps on the cohomology of the entire complex 
in a similar fashion to those produced in plain persistent homology theory. 
Our point of depart is somehow similar, since, in the case the binary operation on the totally ordered set $P$ 
is given by taking minimum, the $A_{\infty}$-algebra structure of the cohomology $H^{\bullet}(\mathcal{D})$ 
of the Rees dg algebra is obtained in essentially the same (but dual) manner as the $A_{\infty}$-coalgebra they consider 
(using the morphisms $f^{i,i+1}$ of that article). 
However, we believe that our manner of proceeding is more systematic and it clarifies the underlying structures, 
for it also indicates that, in order to study the higher multiplications in complete generality 
we should consider instead an intermediate object (see the last paragraph in Subsection \ref{subsection:gen}). 
Furthermore, if the filtration under consideration is better behaved from the point of view of algebra (as the skeletal filtration), our $A_{\infty}$-algebra has more richer structure than (the dual to) that considered in \cite{BM14}.
Finally, in our case, we do not produce any further barcodes but we use the $A_{\infty}$-algebraic structure 
on the complete persistent cohomology group to compare the classic ones.   

On the other hand, the cup product on the persistent cohomology has been studied in the B.Sc. thesis \cite{Yar10} by the student A. Yarmola of M. Vejdemo-Johansson. 
As far as we can tell, our extension of the bottleneck distance does not have any relation whatsoever with the comparison scheme considered in that thesis, 
since the comparison between cup products there is based on the indices of the filtration of the corresponding complexes, 
which are completely independent of the metric 
of the space which should have given origin to them, besides from the fact that their proposal relies only on some consequences of the existence of the cup product. 

The structure of the article is as follows. 
In Section \ref{section:prel}, we provide the basic background on persistent cohomology and on $A_{\infty}$-algebras. 
Since the literature has treated the case of persistent cohomology less frequently than the corresponding homology, we find useful to provide all the definitions, which is done in Subsection \ref{subsection:pers}. 
We have however showed that such distinction is not very stark in some cases. 
Indeed, it follows from a standard construction in homological algebra that, under some usual finiteness and boundedness assumptions, 
the structures of barcodes on bounded intervals of $n$-th persistent homology and of $(n+1)$-th persistent cohomology coincide, 
and there is an equivalence between the restriction of the barcode of $n$-th type for the persistent cohomology to the subset of 
unbounded below (but bounded above) intervals and the restriction of the barcode of $n$-th type for the persistent homology to the subset of 
unbounded above (but bounded below) intervals given by sending a bounded below interval to its complement.  
This has already been observed in \cite{dSMVJ11bis}, for some restricted assumptions, and their proof seems in fact more involved 
than ours in our opinion. 
In any case, we have given a simple proof of this result under more general hypotheses (see Proposition \ref{proposition:dual}). 
Another reason for the exposition is the fact that our definition of persistent cohomology does not really coincide with the one proposed by other authors, as in \cites{dSMVJ11, dSMVJ11bis}, 
but with what they call relative persistent cohomology. 
We have chosen to use the simpler name persistent cohomology, because, from a wider perspective, the difference between 
homological and cohomological dg modules is only conventional.  
On the other hand, we have recalled in Subsection \ref{subsection:kad} the basic definitions on the theory of $A_{\infty}$-algebras together with a result of T. Kadeishvili, which roughly states that the cohomology of any $A_{\infty}$-algebra 
(in particular, any dg algebra) over a field has also a structure of $A_{\infty}$-algebra, such that both are equivalent. 
The $A_{\infty}$-algebra structure on the cohomology is however nonunique, but unique only up to noncanonical equivalence. 

In Section \ref{section:a_infpcoho} we provide the main constructions of the article. 
After briefly indicating in Subsection \ref{subsection:gen} how Kadeishvili's theorem appears naturally in the situation of persistent cohomology under some multiplicative assumptions on the filtration, we provide in Subsection \ref{subsection:bm} an extension of the bottleneck distance on the space of (underlying) $A_{N}$-algebra structure on persistent cohomology, 
for any $N \in \NN \sqcup \{ + \hskip -0.4mm \infty \}$ (see Proposition \ref{proposition:metric}). 
The construction is a little involved but in our opinion fairly natural, as it follows what we believe is the underlying philosophy.  
Moreover, in case $N = 1$, our construction coincides with the usual bottleneck distance (for the complete persistent cohomology, recalled in \eqref{eq:bm}). 
Furthermore, the extended distance is independent of the explicit $A_{N}$-algebra structures chosen for the persistent cohomologies, as one may naturally require.  

Even though we work with the algebraic structure of persistent cohomology, we could have equally worked with the coalgebraic structure of persistent homology, \textit{i.e.} 
by regarding the $A_{\infty}$-coalgebraic structure of the complete persistent homology for a comultiplicative filtration, by using a dual statement of that of Kadeishvili. 
We leave the reader verify that all of our constructions are also valid in that setting \textit{mutatis mutandi}, and the arguments are in fact \textit{verbatim} 
(see the comments at the last paragraph of the introduction of \cite{Her14}). 

%%%%%%%%%%%%%%%%%%%%%%%%%%%%%%%%%%%%%%%%%%%%%%%%%%%%%%%%%%%%%%%%%%%%%%%%%%%%%%%%%%%%%%%%%%%%%%%%%%%%%%%%%%%%%%%%%%%%%%%%%%%%%%%%%%%%%%%%%%%%%%%%
\section{Preliminaries on basic algebraic structures}
\label{section:prel}

In this section we shall recall the basic definitions and constructions on persistent (co)homology and on $A_{\infty}$-algebras. 
With respect to the algebraic structures we shall consider, \textit{e.g.}, algebras, modules, etc., 
we shall follow the same convention on gradings and definitions as those of \cite{Her14a}. 
As explained in that article, these definitions are completely standard and well-known in the literature but they were sometimes presented under minor variations to suit our intentions. 
We refer to that article and the bibliography therein. 
We recall that $k$ will denote a field (which we also consider as a unitary graded ring concentrated in degree zero), 
the term \emph{module} (sometimes decorated by adjectives such as graded, or dg) will denote a symmetric bimodule over $k$ (correspondingly decorated), 
and morphisms between modules will always be $k$-linear (and satisfying further requirements if the modules are decorated as before), unless otherwise stated. 

As usual, all unadorned tensor products $\otimes$ would be over $k$. 
Finally, $\NN$ will denote the set of (strictly) positive integers, whereas $\NN_{0}$ 
will be the set of nonnegative integers. 
Similarly, for $N \in \NN$, we denote by $\NN_{\leq N}$ the set of positive integers less than or equal to $N$. 
Of course, similar notation could be used for other inequality signs.   

%%%%%%%%%%%%%%%%%%%%%%%%%%%%%%%%%%%%%%%%%%%%%%%%%%%%%%%%%%%%%%%%%%%%%%%%%%%%%%%%%%%%%%%%%%%%%%%%%%%%%%%%%%%%%%%%%%%%%%%%%%%%%%%%%%%%%%%%%%%%%%%%
\subsection{\texorpdfstring{Persistent (co)homology}{Persistent (co)homology}} 
\label{subsection:pers}

The notion of persistent homology was introduced more or less simultaneously but in an independent manner by V. Robins in \cite{Ro99}, F. Cagliari, M. Ferri, and P. Pozzi in \cite{CFP01}, 
P. Frosini and C. Landi in \cite{FL99}, and H. Edelsbrunner, D. Letscher and A. Zomorodian in \cite{ELZ02}. 
We refer the reader to the survey \cite{EH08} by the first author of the last cited article and J. Harer, or the comprehensive text \cite{EH10} by the same authors, and also the nice article \cite{CZ05} by 
G. Carlsson and Zomorodian. 

We shall briefly recall a (rather) general scheme for persistent (co)homology. 
Since we are more interested in taking profit from the involved algebraic structures we will use cohomology instead of homology. 
Cohomological variants of the theory have already deserved some attention (see \cites{dSMVJ11, dSMVJ11bis, HB11}). 
Our definition of persistent cohomology does not fully coincide with the corresponding one of the mentioned articles, but with the one considered there under the name of \emph{relative persistent cohomology}.  
We note that we are not considering the more general possible setup, since our persistence modules will be eventually indexed by an essentially discrete totally ordered set $P$, 
but in most of the applications this is already enough general, 
as in the majority of the cases the totally ordered set $P$ is in fact essentially finite.

%%%%%%%%%%%%%%%%%%%%%%%%%%%%%%%%%%%%%%%%%%%%%%%%%%%%%%%%%%%%%%%%%%%%%%%%%%%%%%%%%%%%%%%%%%%%%%%%%%%%%%%%%%%%%%%%%%%%%%%%%%%%%%%%%%%%%%%%%%%%%%%%
\subsubsection{\texorpdfstring{Generalities on gradings and filtrations}{Generalities on gradings and filtrations}} 
\label{subsubsection:gengra}

We recall that a \emph{(cohomological) graded module} over $k$ is a module over $k$ provided with a decomposition of $k$-modules of the form $C = \oplus_{n \in \ZZ} C^{n}$.
If $c$ is a nonzero homogeneous element of a graded module $C$ over $k$ we define the \emph{degree} $\deg c \in \ZZ$ of $c$ by $c \in C^{\deg c}$. 
A \emph{(cohomological) differential graded module (or dg module)} over $k$ is a graded $k$-module $C = \oplus_{n \in \ZZ} C^{n}$ together with a homogeneous $k$-linear map 
$d_{C} : C \rightarrow C$ of degree $+1$, \textit{i.e.} $d_{C}(C^{n}) \subseteq C^{n+1}$ for all $n \in \ZZ$, such that it is a differential, \textit{i.e.} $d_{C} \circ d_{C} = 0$. 
We shall regard a graded module as a dg module with zero differential. 
The homological versions of the previous definitions would be just given by taking lower indices $C = \oplus_{n \in \ZZ} C_{n}$, 
and the differential satisfying instead $d_{C}(C_{n}) \subseteq C_{n-1}$, for all $n \in \ZZ$. 
In fact, all the cohomological definitions can be changed into homological ones (and \textit{vice versa}) by the usual convention $C_{n} = C^{-n}$, for $n \in \ZZ$. 
The distinction between these two notations seems mild for the moment but it is very convenient for explaining several phenomena in a clean manner. 
In any case, we shall most of the time follow the cohomological convention by the reasons stated previously, except in Subsubsection \ref{subsubsection:homocohomo}, where we compare persistent homology and cohomology.  

Let $(P , \succcurlyeq)$ be a partially ordered set (poset). 
A \emph{$P$-filtered dg module} 
will be a dg $k$-module $C = \oplus_{n \in \ZZ} C^{n}$ provided with a (decreasing) filtration $\{ F^{p}C \}_{p \in P}$ of dg submodules, 
\textit{i.e.} $F^{p}C = \oplus_{n \in \ZZ} F^{p}C^{n}$ is a graded module whose $n$-th homogeneous component is $F^{p}C^{n} = F^{p}C \cap C^{n}$, 
$d_{C}(F^{p}C) \subseteq F^{p}C$, and $F^{p}C \supseteq F^{p'}C$, for all $p, p' \in P$ such that $p \preccurlyeq p'$. 
Since the poset will be clear from the context when dealing with dg modules 
provided with filtrations, we shall usually omit the former from the notation. 
If the dg module $(C,d_{C})$ was homological, one would consider increasing filtrations of dg modules instead.
We say that the filtration of $C$ is \emph{exhaustive} if $\cup_{p \in P} F^{p}C = C$, and that it is \emph{Hausdorff} if $\cap_{p \in P} F^{p}C = \{ 0 \}$. 
A particular case of interest is when the filtration is \emph{bounded}, \textit{i.e.} there exists $p', p \in P$ such that $p' \succcurlyeq p$
satisfying that $F^{p'}C = C$ and $F^{p'}C = 0$. 
Moreover, we say it is \emph{essentially discrete (resp., finite)}, 
\textit{i.e.} there exists a  subposet
$P_{s}$ of the poset $P$ which is a discrete subset of the topological space $P$ 
provided with the order topology (resp., a finite subset of $P$) and satisfying that for all $p \in P$ there exists $p' \in P_{s}$ satisfying that $F^{p}C = F^{p'}C$.   

%%%%%%%
\begin{remark}
\label{remark:hoco}
The conventions between the increasing filtrations in the case of homological dg modules 
and decreasing ones in the cohomological case should not cause much confusion, 
and they are in fact related as follows. 
We first recall that the \emph{opposite poset} $P^{\mathrm{op}}$ of a poset $P$ 
has the same underlying set as $P$ but with the order $\succcurlyeq^{\mathrm{op}}$ defined as 
$p \succcurlyeq^{\mathrm{op}} p'$ if an only if $p' \succcurlyeq p$.  
If $C^{\bullet}$ is a cohomological dg module provided with a decreasing $P$-filtration $\{F^{p}C\}_{p \in P}$, then the corresponding homological dg module $C_{\bullet}$ 
(which we recall it is defined as $C_{n} = C^{-n}$, for $n \in \ZZ$) has an increasing filtration over the opposite poset $P^{\mathrm{op}}$ 
given by $F_{p}C = F^{p}C$, where the first $p$ is in the poset $P^{\mathrm{op}}$, whereas the second is in $P$. 
Of course this construction is clearly applied in the opposite direction from increasingly filtered homological dg modules to decreasingly filtered cohomological ones. 
\end{remark}
%%%%%%%

Let $(C,d_{M})$ be a filtered dg module for a filtration $\{ F^{p}C \}_{p \in P}$. 
This situation naturally appears if $X$ is for instance a topological space provided with an increasing filtration $\{ X_{p} \}_{p \in P}$ of topological subspaces, 
and one considers the complex $C^{\bullet}(X,k)$ computing singular cohomology which is provided with a decreasing filtration of $\{ F^{p}C^{\bullet}(X,k) \}_{p \in P}$, 
where $F^{p}C^{\bullet}(X,k)$ is the kernel of the canonical map $C^{\bullet}(X,k) \rightarrow C^{\bullet}(X_{p},k)$, which is the dual of the inclusion mapping $C_{\bullet}(X_{p},k) \rightarrow C_{\bullet}(X,k)$. 
We could also consider instead a simplicial set $\mathcal{X}$, provided with an increasing filtration $\{ \mathcal{X}_{p} \}_{p \in P}$ of simplicial sets, 
and we regard thus the complex $C^{\bullet}(\mathcal{X},k)$ of normalized cochains (\textit{i.e.} the maps $f : \mathcal{X}_{n} \rightarrow k$ that vanish on degenerate cycles) 
computing the simplicial cohomology, provided with a decreasing filtration $\{ F^{p}C^{\bullet}(\mathcal{X},k) \}_{p \in P}$, 
where $F^{p}C^{\bullet}(\mathcal{X},k)$ is the set of normalized cochains vanishing on $\mathcal{X}_{p}$ (\textit{cf.} \cite{FHT01}, Ch. 10, (d)). 
Furthermore, one is typically interested in the situation that either the topological space or the simplicial set are embedded in $\RR^{n}$, so they are in fact metric spaces, 
and the filtration is constructed by means of the metric. 
If one forgets about the embedding into $\RR^{n}$, the first situation is in fact a particular case of the first one, 
since one may consider $\mathcal{X}$ to be the simplicial set $S_{\bullet}(X)$ of singular chains of $X$. 
A typical assumption in the theory is that the filtration is \emph{Hausdorff}, \textit{i.e.} $\cap_{p \in P} X_{p} = \emptyset$ (resp., $\cap_{p \in P} \mathcal{X}_{p} = \emptyset$), 
and \emph{exhaustive}, \textit{i.e.} $\cup_{p \in P} X_{p} =  X$ (resp., $\cup_{p \in P} \mathcal{X}_{p} = \mathcal{X}$). 
This in turn implies that the filtration of the corresponding cochain complex computing cohomology of $X$ or $\mathcal{X}$ is exhaustive or Hausdorff, respectively, as well. 
Moreover, the definitions of bounded, essentially discrete or finite make perfect sense for the increasing filtrations of sets 
mentioned before as well. 
From now on, we shall assume that all the filtrations are Hausdorff and exhaustive. 

%%%%%%%
\begin{remark}
\label{remark:vr}
One of the most considered filtered dg modules in topological data analysis is the \emph{Vietoris-Rips complex} $\operatorname{VR}(X)$ for a finite set $X \subseteq \RR^{n}$ provided with the induced metric. 
In that case the (homological) dg module is given by the chain complex $\operatorname{VR}(X) = C_{\bullet}(\mathcal{X},k)$, where $\mathcal{X}$ is the simplicial set associated to the combinatorial simplicial complex $(X , \mathscr{P}(X) \setminus \{ \emptyset \})$ (fixed an ordering), where $\mathscr{P}(X)$ denotes the set of subsets of $X$ (see \cite{Wei94}, Example 8.1.8). 
The (increasing) $\RR$-filtration $G_{\epsilon}\operatorname{VR}(X)$, for $\epsilon \in \RR$, is defined as $C_{\bullet}(G_{\epsilon}\mathcal{X},k)$, where $G_{\epsilon}\mathcal{X}$ is the simplicial set associated to the combinatorial subsimplicial complex of $(X , \mathscr{P}(X) \setminus \{ \emptyset \})$ 
formed by the the simplices $\sigma = \{ x_{0}, \dots, x_{i} \}$ such that $d(x_{j},x_{l}) < \epsilon$, 
for all $j, l \in \{1, \dots, i \}$. 
For further details, together with the definition of other filtrations of the kind using the metric, 
as the \emph{\v{C}ech filtration} or what is called the \emph{$\alpha$-complex}, 
we refer to \cite{EH10}, Ch. III, Sections 2--4. 
We would like to remark however that the possible filtrations considered in topological data analysis 
are somehow restricted (as before) to be constructed in a similar fashion, heavily depending on the metric. 
Even though this point may seem awkward, the rigidity on the filtrations to be chosen is however 
important in principle if one is interested in some algebraic properties, since the mentioned filtrations, 
supposed to be essentially finite and thus may be even reindexed with a finite subset of $\ZZ$, 
are not necessarily going to be \emph{comultiplicative} with respect to the standard coalgebra structure 
on the corresponding chain of complexes in any sensible manner. 
Indeed, by choosing positive large indices for the filtration one may always get that the comultiplicative structure on the associated spectral sequence is trivial. 
\end{remark} 
%%%%%%%

%%%%%%%%%%%%%%%%%%%%%%%%%%%%%%%%%%%%%%%%%%%%%%%%%%%%%%%%%%%%%%%%%%%%%%%%%%%%%%%%%%%%%%%%%%%%%%%%%%%%%%%%%%%%%%%%%%%%%%%%%%%%%%%%%%%%%%%%%%%%%%%%
\subsubsection{\texorpdfstring{From exact couple systems to persistent cohomology}{From exact couple systems to persistent cohomology}} 
\label{subsubsection:ecs}

We will now recall the definition of an exact couple system 
over a bounded poset $P$, \textit{i.e.} a poset having a minimum and a maximum, 
which we shall denote by $\mathrm{min}(P)$ and $\mathrm{Max}(P)$, respectively. 
This definition is a generalization recently introduced by B. Matschke in his preprint article \cite{Ma13} of the usual notion of exact couple given by W.S. Massey in \cite{Ma52}. 
We remark that we use a slightly different indexing convention, which is more convenient in our setting. 
Consider first the graded category $P_{2}$, whose set of objects $\mathrm{Ob}(P_{2})$ is formed by the pairs $(p',p)$ satisfying that $p' \preccurlyeq p$ and with the following morphisms. 
For each pair of objects $(p'_{1}, p_{1}), (p'_{2}, p_{2})$ in $P_{2}$ 
satisfying that $p'_{1} \succcurlyeq p'_{2}$ and $p_{1} \succcurlyeq p_{2}$ we have a unique morphism 
\[     \phi_{p'_{1},p_{1}}^{p'_{2},p_{2}} : (p'_{1}, p_{1}) \rightarrow (p'_{2}, p_{2})     \] 
of degree $0$, and for each three elements $p', p, p''$ in $P$ such that $p', p'' \preccurlyeq p$ we have a unique morphism 
\[     \psi_{p',p,p''} : (p', p) \rightarrow (p'', \mathrm{Max}(P))     \] 
of degree $1$. 
There are no other morphisms, and the compositions (and identities) are thus uniquely determined 
(this implies in particular the identities at the first paragraph of \cite{Ma13}, Rmk. 4.2). 
A \emph{cohomological exact couple system} over the bounded poset $P$ is a (degree preserving covariant) functor $\mathcal{E}$ from the graded category $P_{2}$ to the (graded) category of graded $k$-modules satisfying that, for each object $(p',p)$ 
in $P_{2}$, the triangle 
\[
\xymatrix
{
D^{p}
\ar[rr]^{\mathrm{i}^{p',p}}
&
&
D^{p'}
\ar[dl]^{\mathrm{j}^{p',p}}
\\
&
E^{p'}_{p}
\ar[ul]^{\mathrm{k}^{p',p}}
&
}
\]
is exact at each vertex, where we have used the (standard) notation given by
\begin{equation}
\begin{split}
D^{p} &= \mathcal{E}(p, \mathrm{Max}(P)), \hskip 1cm E^{p'}_{p} = \mathcal{E}(p',p), 
\\
\mathrm{i}^{p',p} = \mathcal{E}(\phi^{p',\mathrm{Max}(P)}_{p,\mathrm{Max}(P)}),  &\hskip 1cm 
\mathrm{j}^{p',p} = \mathcal{E}(\phi^{p',p}_{p',\mathrm{Max}(P)}), \hskip 1cm
\mathrm{k}^{p',p} = \mathcal{E}(\psi_{p',p,p}). 
\end{split}
\end{equation}
For all $p', p, p'' \in P$ such that $p'' \succcurlyeq p \succcurlyeq p'$ 
one may define the morphism $d^{p',p, p''} = \mathrm{j}^{p,p''} \circ \mathrm{k}^{p',p}$ from $E^{p'}_{p}$ to $E^{p}_{p''}$. 
They satisfy a differential property, \textit{i.e.} $d^{p,p'',p'''} \circ d^{p',p,p''} = 0$, for any  $p', p, p'', p''' \in P$ such that $p''' \succcurlyeq p'' \succcurlyeq p \succcurlyeq p'$.  
The indexing is fairly similar (but not exactly) to the one usually used for (classical) exact couples. 

Now, let $(C,d_{C})$ be a dg module over $k$ provided with a decreasing filtration $\{ F^{p}C \}_{p \in P}$ of the underlying graded module of $C$ such that $d_{C}(F^{p} C) \subseteq F^{p}C$, for all $p \in P$. 
We recall that we are assuming that the filtration is exhaustive and Hausdorff. 
Let us consider the bounded poset $\hat{P} = P \sqcup \{ - \infty, + \hskip -0.4mm \infty \}$, 
extending the order of $P$, and where $- \infty$ is the infimum of $\hat{P}$ and $+ \hskip -0.4mm \infty$ the supremum. 
Set $F^{+ \hskip -0.4mm \infty}C = 0$ and $F^{- \infty}C = C$. 
We define the dg $k$-modules 
\begin{equation}
\label{eq:reescomplex}
     \mathcal{D} = \bigoplus_{p \in \hat{P}} F^{p}C,     
\end{equation}
provided with the differential which sends $z \in F^{p}C$ to $d_{C}(z) \in F^{p}C$, for all $p \in \hat{P}$, and 
\begin{equation}
\label{eq:gr}  
   \mathcal{G} = \bigoplus_{(p',p) \in \mathrm{Ob}(\hat{P}_{2})} \frac{F^{p'}C}{F^{p}C},     
\end{equation} 
with the induced differential by $d_{C}$, \textit{i.e} 
it sends the class $z + F^{p}C$ to $d_{C}(z) + F^{p}C$, 
for $z \in F^{p'}C$. 
Furthermore, these dg modules are provided with bigradings. 
In the first dg module, the one degree comes from the poset $\hat{P}$, and will be called the \emph{filtration degree}, and the other comes 
from the cohomological grading of $C$, and will be still called \emph{cohomological degree}.  
We set thus 
\[     \mathcal{D}^{p,n} = F^{p}C^{n}.     \] 
The second dg module has one degree coming from the direct sum indexed by the set $\mathrm{Ob}(\hat{P}_{2})$, which may be regarded as a poset for the 
\emph{product order}, \textit{i.e.} $(p'_{1},p_{1}) \succcurlyeq (p'_{2},p_{2})$ if and only if $p'_{1} \succcurlyeq p'_{2}$ and $p_{1} \succcurlyeq p_{2}$. 
The second degree also comes from the cohomological grading of $C$. 
We shall write it as
\[     \mathcal{G}^{p',n}_{p} = (F^{p'}C^{n})/(F^{p}C^{n}).     \] 
With respect to their bigrading, it is easy to see that the differential of both $\mathcal{D}$ and $\mathcal{G}$ have cohomological degree $1$ and preserve the other degree.  
The dg module $\mathcal{D}$ will be called the \emph{Rees dg module associated to $C$}, 
whereas $\mathcal{G}$ is typically 
called the \emph{associated graded dg module of $C$}. 
From the previous dg modules we immediately obtain an exact couple system, which is called the \emph{exact couple system associated to the filtered dg module $C$}, as follows. 
More precisely, for every object $(p',p)$ in $\hat{P}_{2}$ define $E^{p'}_{p} = H^{\bullet}(F^{p'}C/F^{p}C)$ 
(this implies that for every $p \in \hat{P}$, $D^{p} = H^{\bullet}(F^{p}C)$), 
and for any pair of objects $(p'_{1}, p_{1}), (p'_{2}, p_{2})$ in $\hat{P}_{2}$ 
satisfying that $p'_{1} \succcurlyeq p'_{2}$ and $p_{1} \succcurlyeq p_{2}$, define 
$\mathcal{E}(\phi_{p'_{1},p_{1}}^{p'_{2},p_{2}}) : H^{\bullet}(F^{p'_{1}}C/F^{p_{1}}C) \rightarrow H^{\bullet}(F^{p'_{2}}C/F^{p_{2}}C)$ as the cohomology of the unique map of dg modules 
$F^{p'_{1}}C/F^{p_{1}}C \rightarrow F^{p'_{2}}C/F^{p_{2}}C$ induced by the canonical inclusions. 
For $p', p, p''$ in $\hat{P}$ such that $p', p'' \preccurlyeq p$, the map $\mathcal{E}(\psi_{p',p,p''})$ 
is given by the composition of the map $H^{\bullet}(F^{p'}C/F^{p}C) \rightarrow H^{\bullet}(F^{p}C)$ 
given by the Snake lemma applied to the short exact sequence of dg modules 
\[     0 \rightarrow F^{p}C \rightarrow F^{p'}C \rightarrow F^{p'}/F^{p}C \rightarrow 0     \] 
(see \cite{Wei94}, Thm. 1.3.1, Example 5.9.3), together with the map $H^{\bullet}(F^{p}C) \rightarrow H^{\bullet}(F^{p''}C)$ given by the cohomology of the canonical inclusion of dg modules $F^{p}C \rightarrow F^{p''}C$. 
It is straightforward to check that this gives an exact couple system, 
which may be depicted as follows  
\[
\xymatrix
{
H^{\bullet}(\mathcal{D}^{p}) 
\ar[rr]^{\mathrm{i}^{p',p}}
&
&
H^{\bullet}(\mathcal{D}^{p'})
\ar[dl]^{\mathrm{j}^{p',p}}
\\
&
H^{\bullet}(\mathcal{G}^{p'}_{p})
\ar[ul]^{\mathrm{k}^{p',p}}
&
}
\]

As usual, given a poset $P$, we may consider the category $\tilde{P}$ (also denoted $P^{\sim}$)
whose objects are given by the elements of $P$ and with the following morphisms. 
For every pair of elements $p', p \in P$ satisfying that $p \succcurlyeq p'$, we define a unique morphism 
$\omega^{p'}_{p} : p \rightarrow p'$. 
There are no other morphisms, and the compositions (and identities) are thus uniquely determined. 
This category is sometimes called the \emph{category associated to the poset $P$}. 
We see now that the dg module $\mathcal{D}$ is a module over the category $\tilde{P}$ as well. 
In fact, considering $k.\tilde{P}$, which is the graded category with the same objects as $\tilde{P}$ and 
whose space of morphisms from $p$ to $p'$ is given by the free $k$-module generated by $\omega^{p'}_{p}$, 
considered as a graded module concentrated in degree zero, 
we may regard $\mathcal{D}$ as a degree preserving covariant functor $\tilde{\mathcal{D}}$ from $k.\tilde{P}$ to the (graded) category of complexes of modules over $k$ in the canonical manner. 
We remark first that the latter category has as objects the dg modules over $k$ but the spaces of homomorphisms are given by the corresponding 
sets of morphisms of dg modules of degree zero 
commuting with the differentials. 
For every $p \in P$ we set $\tilde{\mathcal{D}}(p) = \mathcal{D}^{p}$, and for 
$\omega^{p'}_{p}$ we set $\tilde{\mathcal{D}}(\omega^{p'}_{p})$ as the canonical inclusion 
$\mathcal{D}^{p} \rightarrow \mathcal{D}^{p'}$ of dg modules. 
By taking cohomology pointwise, 
\textit{i.e.} by considering the composition $H^{\bullet} \circ \tilde{\mathcal{D}}$, 
we get a covariant functor from $k.\tilde{P}$ to the category of graded modules over $k$ (provided only with morphisms of degree $0$). 
In other words, we may regard the cohomology of $\mathcal{D}$ as a graded module over $k.\tilde{P}$, 
which will be called the \emph{complete persistent cohomology} of the filtered dg module $C$.
We shall also denote it by $H^{\bullet}(\mathcal{D})$. 
Analogously, for each $n \in \ZZ$, the $n$-th cohomology of $\mathcal{D}$ can be regarded as a module over $k.\tilde{P}$, denoted by $H^{n}(\mathcal{D})$ and called the \emph{$n$-th persistent cohomology of the filtered dg module $C$}, and we see easily that 
\[     H^{\bullet}(\mathcal{D}) \simeq \bigoplus_{n \in \ZZ} H^{n}(\mathcal{D}).     \]
Given $p' , p \in P$ such that $p \succcurlyeq p'$, the \emph{$n$-th persistent cohomology 
$PH^{n}_{p,p'}(C)$ of the filtered dg module $C$ starting at $p$ and ending at $p'$} 
is the image of the morphism $\omega^{p'}_{p}$ under the functor $H^{n}(\mathcal{D})$. 
In other words, it is the image of the map 
\[     H^{n}(F^{p}C) \rightarrow H^{n}(F^{p'}C)     \]
induced by the canonical inclusion $F^{p}C \rightarrow F^{p'}C$. 

The structure of the previous module(s) is in general too complicated to be dealt with for a 
general poset $P$. 
For the rest of this subsection, we shall assume that $P$ is totally ordered and such that its order topology is \emph{separable}, 
\textit{i.e.} it has a countable dense subset. 
In this case, for every $n \in \ZZ$ the module $H^{n}(\mathcal{D})$ over $k.\tilde{P}$ is a \emph{(left) persistence module} (on $P$), 
\textit{i.e.} a $k$-linear covariant functor from $k.\tilde{P}$ to the category of $k$-modules. 
We shall usually omit the adjective left in order to simplify the notation. 
Of course, morphisms of persistence modules are just natural transformations between the respective functors. 
For any persistence module $M$, its \emph{dimension function} 
is the map $P \rightarrow \NN_{0}$ 
given by sending $p$ to $\dim M(p)$.

The idea of considering (small $k$-linear) categories as generalizations of algebras over $k$ and ($k$-linear) functors from them to the category of $k$-modules as generalizations of modules is not new, but it is perhaps not 
so well-known, and initiated by P. Freyd in his thesis \cite{Fre60} and was further extended by B. Mitchell in 
\cite{Mi72}, and others. 
In fact, since small $k$-linear categories are equivalent to (nonunitary) $k$-algebras (satisfying some local unit conditions), such that the $k$-linear functors correspond to modules 
(also satisfying some local unit conditions), we may see that all elementary constructions for $k$-algebras and their modules are possible for $k$-linear categories and $k$-linear functors on them as well.
We refer the reader to the last mentioned bibliography for further details. 

%%%%%%%
\begin{remark}
\label{remark:coho2}
If we were interested in increasing $P$-filtrations of homological dg modules, 
we would consider instead right persistence modules, \textit{i.e.} $k$-linear 
contravariant functors from $k.\tilde{P}$ to the category of $k$-modules. 
We note that if $M$ is a right persistence module, then the composition 
of $M$ with the functor $(\place)^{*}$ given by taking the dual is a left persistence module, 
and \textit{vice versa}. 
Notice also that considering right persistence modules is equivalent to consider left persistence 
modules over $k.(P^{\mathrm{op}})^{\sim} \simeq (k.\tilde{P})^{\mathrm{op}}$. 
\end{remark}
%%%%%%%

As usual, if $M$ is a persistence module and $m \in M(p)$, for $p \in P$, 
we may denote $M(\omega^{p'}_{p})(m)$ simply by $\omega^{p'}_{p}.m$. 
Given a persistence module $M$, one may define its \emph{torsion submodule} $M_{\mathrm{tor}}$ 
given by 
\[     M_{\mathrm{tor}}(p) = \{ m \in M(p) : \text{there exists $p' \prec p$ such that $\omega^{p'}_{p}.m$ vanishes } \}.     \]     
A persistence module $M$ is called \emph{pointwise finite dimensional} if $M(p)$ is a finite dimensional 
$k$-module for every $p \in P$. 
A subset $I \subseteq P$ is called an \emph{interval} if it is nonempty and satisfies that, 
given $p', p'' \in I$ and $p \in P$ such that $p' \preccurlyeq p \preccurlyeq p''$, then $p \in I$. 
An interval is said to be \emph{bounded below (resp., above)} if there exists $p \in P \setminus I$ and  
$p' \in I$ such that $p \prec p'$ (resp., $p \succ p'$). 
An interval which is not bounded below (resp., above) will be called \emph{unbounded below (resp., above)}. 
An interval is called \emph{bounded} if it bounded above and below. 
Given an interval $I \subseteq P$, the corresponding \emph{interval (left) module} $k_{I}$ over $k.\tilde{P}$ 
is defined as follows. 
Set $k_{I}(p) = k$ if $p \in I$ and otherwise $k_{I}(p) = 0$. 
Moreover, for $p', p \in I$ define $k_{I}(\omega^{p',p})$ to be the identity of $k$. 
If either $p'$ or $p$ does not belong to $I$, $k_{I}(\omega^{p',p})$ is necessarily the zero map. 
Note that $(k_{I})_{\mathrm{tor}} = 0$ if and only if $I$ is bounded below, 
otherwise $(k_{I})_{\mathrm{tor}} = k_{I}$.  
It is clear to see that every interval module is indecomposable, since its endomorphism ring is just $k$, a local ring (see \cite{AF92}, Thm. 12.6), and two interval modules $k_{I}$ and $k_{I'}$ are isomorphic if and only if $I = I'$.
The importance of interval modules comes from the following result, proved by W. Crawley-Boevey 
(see \cite{CB13}, Thm. 1.1). 
The uniqueness statement follows immediately from the Krull-Schmidt-Azumaya theorem (see \cite{AF92}, Thm. 12.6). 
%%%%%%%
\begin{theorem}
\label{theorem:cb}
Any pointwise finite dimensional persistence module is a direct sum of interval modules. 
This decomposition is thus unique up to permutation of the interval modules. 
\end{theorem} 
%%%%%%%

It is easy to characterize morphisms between interval modules. 
Let $I, J \subseteq P$ be two intervals. 
We say that $I$ has \emph{good intersection} with $J$ 
if $I \cap J \neq \emptyset$, and they satisfy that, for all $p \in I$, there exists $p' \in J$ 
such that $p' \succcurlyeq p$, and for all $p' \in J$ there exists $p \in I$ such that $p' \succcurlyeq p$. 
We note that the order of the intervals in the statement of the definition is important. 
However, we may also say that $I$ and $J$ have good intersection, if this causes no confusion. 
The following result is direct. 
%%%%%%%
\begin{fact}
\label{fact:1}
Given two interval modules $k_{I}$ and $k_{J}$, we have a nonzero morphism from $k_{I}$ to $k_{J}$ if and only if $I$ and $J$ have good intersection, 
in which case all the homomorphism from the former to the latter persistence modules 
are given by a multiplicative factor of the morphism $\phi_{I,J} : k_{I} \rightarrow k_{J}$
defined as $\phi_{I,J}(p) = \mathrm{id}_{k}$, if $p \in I \cap J$, and $\phi_{I,J}(p) = 0$ otherwise.
\end{fact}
%%%%%%%

For a totally ordered set $P$, we recall that $\hat{P}$ is 
the poset given by $P \sqcup \{  - \infty, + \hskip -0.4mm \infty \}$, where we set $- \infty \prec p \prec + \hskip -0.4mm \infty$, 
for all $p \in P$, and the elements of $P$ inside $\hat{P}$ are ordered as in $P$. 
We say that a filtration over a totally ordered set $P$ 
is \emph{continuous on the left (resp., right)} if there exists a 
collection of intervals $\{ K_{a} \}_{a \in A}$ of $P$ 
such that each $K_{a}$ is of the form 
$\{ p \in P : s_{a} \prec p \preccurlyeq t_{a} \}$ (resp., $\{ p \in P : s_{a} \preccurlyeq p \prec t_{a} \}$), for some 
$s_{a}, t_{a} \in \hat{P}$ such that $s_{a} \prec t_{a}$, 
and satisfies that $K_{a} \cap K_{a'} = \emptyset$ for all $a \neq a'$ in $A$, 
and $\cup_{a \in A} K_{a} = P$, and the filtration fulfils that $F^{p}C = F^{p'}C$, 
if $p, p' \in K_{a}$, for some $a \in A$. 
All these definitions apply to increasing filtrations of homological dg modules, topological spaces or simplicial spaces as well. 
In particular, note that the Vietoris-Rips complex in Remark \ref{remark:vr} is continuous on the left. 

%%%%%%%
\begin{remark}
\label{remark:dualint}
Note that the dual of a left persistence module of the form $\oplus_{a \in A} k_{I_{a}}$, for a family of intervals $\{ I_{a} \}_{a \in A}$, 
is a right persistence module $\oplus_{a \in A} k_{I_{a}}^{*}$, where $k_{I}^{*}$ is the corresponding interval right module. 
We could also have chosen to write the latter as $k_{I}$, but we should indicate that the (obvious) right action is intended. 
\end{remark}
%%%%%%%

%%%%%%%%%%%%%%%%%%%%%%%%%%%%%%%%%%%%%%%%%%%%%%%%%%%%%%%%%%%%%%%%%%%%%%%%%%%%%%%%%%%%%%%%%%%%%%%%%%%%%%%%%%%%%%%%%%%%%%%%%%%%%%%%%%%%%%%%%%%%%%%%
\subsubsection{\texorpdfstring{Barcodes}{Barcodes}} 
\label{subsubsection:bar}

We recall that a \emph{multiset (or bag)} $\mathcal{S}$ (see \cite{Sta04}, Ch. 1, Section 2) consists of a set $S$ and a multiplicity map $v_{S} : S \rightarrow (\NN_{0} \sqcup \{ +\hskip -0.4mm \infty \})$. 
The value $v_{S}(s)$ is called the \emph{multiplicity of $s$}, and the set $| \mathcal{S} |$ formed by the elements $s$ of $S$ such that $v_{S}(s) \neq 0$ is called the \emph{support of $\mathcal{S}$}. 
We shall regard the codomain $\NN_{0} \sqcup \{ +\hskip -0.4mm \infty \}$ of the multiplicity map as a totally ordered set, with the usual order of $\NN_{0}$ and such that $n < + \hskip -0.4mm \infty$, 
for all $n \in \NN_{0}$, and we shall denote the set of all multisets defined on a fixed set $S$ by $\operatorname{Multi}(S)$. 
We may see a set $S$ as a multiset for the multiplicity map $v_{S} : S \rightarrow (\NN_{0} \sqcup \{ +\hskip -0.4mm \infty \})$ given by the constant function $s \mapsto 1$, for all $s \in S$. 
Given two multisets $\mathcal{S} = (S,v_{S}) $ and $\mathcal{S}' = (S',v_{S'})$, we say that they are \emph{equivalent} if there exists a bijection $f : |\mathcal{S}| \rightarrow |\mathcal{S}'|$ such that 
$v_{S'} \circ f = v_{S}|_{|\mathcal{S}|}$. 
Moreover, a multiset $\mathcal{S}$ is called \emph{finite} if the image of $v_{S}$ lies in $\NN_{0}$ and $\sum_{s \in v^{-1}(\NN)} v(s)$ is finite. 
The latter number is called the \emph{cardinality (size, or number of elements)} of the finite multiset. 
If the multiset is not finite, we say that its cardinality is infinite. 
Let $\mathcal{S} = (S,v_{S})$ and $\mathcal{S}' = (S',v_{S'})$ be two multisets. 
Take $*$ an element disjoint to $S$ and $S'$, and consider the disjoint unions $\hat{S} = S \sqcup \{ * \}$ and $\hat{S}' = S' \sqcup \{ * \}$. 
A \emph{(partial) matching} between $\mathcal{S}$ and $\mathcal{S}'$ is given by a multiset $\mathcal{P}$ defined on the set 
$(\hat{S} \times \hat{S}') \setminus \{ (*,*) \}$ with multiplicity $v_{P}$ satisfying that, for each $s \in S$, we have that  
\[     \sum_{\hat{s}' \in \hat{S}'} v_{P}(s,\hat{s}') = v_{S}(s),     \]
and for each $s' \in S'$, we have that
\[     \sum_{\hat{s} \in \hat{S}} v_{P}(\hat{s},s') = v_{S'}(s').     \]
We shall denote by $\operatorname{PM}(\mathcal{S},\mathcal{S}')$ the set of partial matchings between $\mathcal{S}$ and $\mathcal{S}'$. 
Let us define
\[     \mathcal{I}_{P} = \{ I : \text{ $I \subseteq P$ is an interval} \}.      \] 
A multiset whose underlying set is $\mathcal{I}_{P}$ would be called a 
\emph{barcode of $P$}. 

Given now any pointwise finite dimensional persistence module $H$, by Theorem \ref{theorem:cb}, we may decompose it as a direct sum of interval modules. 
For every interval $I \subseteq P$, let $\mu_{I}(H)$ be the multiplicity of the interval module $k_{I}$ in 
this decomposition, \textit{i.e} the cardinality of the set of interval modules appearing in the (or any) decomposition of $H$ into interval modules which are isomorphic to $k_{I}$. 
We may also be interested in the case that $H$ is not necessarily pointwise finite dimensional, 
but it has a decomposition as in Theorem \ref{theorem:cb} (which, by the Krull-Schmidt-Azumaya, is thus unique up to 
permutation of the factors. See \cite{AF92}, Thm. 12.6). 
So, we allow the mentioned cardinality to be infinite and we set in those cases $\mu_{I}(H) = + \hskip -0.4mm \infty$. 
In any case, the \emph{barcode associated to $H$} has as underlying set $\mathcal{I}_{P}$ 
together with the map $v_{H} : \mathcal{I}_{P} \rightarrow (\NN_{0} \sqcup \{ +\hskip -0.4mm \infty \})$ defined by $v_{H}(I) = \mu_{I}(H)$. 
The associated multiset to $H^{n}(\mathcal{D})$ (provided it is pointwise finite dimensional) 
is called the \emph{barcode of $n$-th type associated to the filtered complex $C$}. 
We shall denote in this case the corresponding multiplicity number by $\mu^{n}_{I}(C)$, or just $\mu^{n}_{I}$ if the filtered dg module $C$ is clear from the context. 
If the interval $I$ is of the form $\{ p \in P : p'' \succcurlyeq p \succ p' \}$, for some $p'' \succ p'$ in $P$, 
we shall even denote the previous multiplicity by $\mu^{n}_{p'',p'}(C)$ (or just $\mu^{n}_{p'',p'}$), 
and if it is of the form $\{ p \in P : p'' \succcurlyeq p \}$, for some 
$p'' \in P$, then we shall denote the corresponding multiplicity by $\mu^{n}_{p''}(C)$ (or just $\mu^{n}_{p''}$)

In order to put all these barcodes of $n$-th type together we shall do the following construction. 
We consider the set 
\[     \mathcal{GI}_{P} = \{ (n, I) : \text{ $n \in \ZZ$ and $I \subseteq P$ is an interval} \},     \] 
which we are going to call the set of \emph{generalized intervals}. 
An element $(n,I)$ is a \emph{generalized interval of $n$-th type}. 
A \emph{generalized barcode} is an element of $\operatorname{Multi}(\mathcal{GI}_{P})$, \textit{i.e.} a multiset whose underlying set is $\mathcal{GI}_{P}$. 
To any filtered dg module $(C,d_{C})$ satisfying the same assumptions as in the previous paragraph we can assign a generalized barcode $\mathcal{S}(C)$, 
called the \emph{generalized barcode associated to $C$}, whose multiplicity map $v(C)$ sends $(n,I)$ to $\mu^{n}_{I}(C) = v_{H^{n}(\mathcal{D})}(I)$, 
where we recall that $\mathcal{D}$ is given by \eqref{eq:reescomplex}. 

A typical interpretation of these generalized barcodes in case the interval $I$ is of the form $\{ p \in P : p'' \succcurlyeq p \succ p' \}$, for some $p'' \succ p'$ in $P$, is that for each $n$ there is (a basis of) 
$\mu^{n}_{p'',p'}$ $n$-th cohomology classes which originate at stage $p''$ and persist till they extinguish at stage $p'$. 
If $I = \{ p \in P : p'' \succcurlyeq p \}$, for some 
$p'' \in P$, then for each $n$ there is (a basis of) 
$\mu^{n}_{p''}$ $n$-th cohomology classes which originate at stage $p''$ and are not extinguished.
The idea of persistent (co)homology is that (co)homology classes which persistent longer, \textit{i.e.} for which the interval $I$ is larger, are more meaningful than those persisting less.  

%%%%%%%
\begin{remark}
\label{remark:essfin}
The importance of the types of intervals considered in the previous paragraphs stems from the fact that any interval with nonzero multiplicity in the 
barcode associated to a filtered dg module with an essentially discrete and continuous on the left filtration is of the form prescribed there. 
Even more, there exists a discrete subset $P_{s} \subset P$ satisfying that, for any $p \in P$ there exists $p' \in P_{s}$ such that $F^{p}C = F^{p'}C$ and the possible intervals having nonzero multiplicity in the associated barcode 
are of the form $I = \{ p \in P : p'' \succcurlyeq p \}$, for some 
$p'' \in P_{s}$, 
or $\{ p \in P : p'' \succcurlyeq p \succ p' \}$, for some $p'' \succ p'$ in $P_{s}$. 
The proof of this result follows easily from Theorem \ref{theorem:cb}.  
\end{remark}
%%%%%%%

Suppose that the filtration of $C$ is essentially finite, continuous on the left and of locally finite dimensional cohomology, \textit{i.e.} $H^{n}(F^{p}C)$ is finite dimensional for all $n \in \ZZ$ and $p \in P$. 
Hence, the persistence module $H^{n}(\mathcal{D})$ is pointwise finitely dimensional 
and the barcode construction of the fourth previous paragraph applies. 
Let $P_{s} \subset P$ be a finite subposet as in Remark \ref{remark:essfin}.  
Then, all the intervals with nonzero multiplicity in the associated 
barcode are of the form $\{ p \in P : p'' \succcurlyeq p \succ p' \}$, for some $p'' \succ p'$ in $P_{s}$, 
or $\{ p \in P : p'' \succcurlyeq p \}$, for some $p'' \in P_{s}$. 
We recall the well-known fact that, under these assumptions, the barcodes and the dimensions of the persistent cohomology groups of a filtered complex $C$ contain in fact the ``same'' information. 
Indeed, let $\beta^{n}_{p,p'}$ denote the dimension of $PH^{n}_{p,p'}(C)$, and we shall denote in what 
follows $p + 1$ the immediate successor of $p \in P_{s}$, if $p \neq \mathrm{Max}(P_{s})$, and 
$p+1 = \mathrm{Max}(P_{s})$ if $p = \mathrm{Max}(P_{s})$. 
It is then straightforward to prove that 
\begin{equation}
\label{eq:mul1}
     \mu^{n}_{p,p'} = \beta^{n}_{p+1,p'} - \beta^{n}_{p,p'} + \beta^{n}_{p,p'+1} - \beta^{n}_{p+1,p'+1},     
\end{equation}
and 
\begin{equation}
\label{eq:mul2}
     \mu^{n}_{p} = \beta^{n}_{p,\mathrm{min}(P_{s})} - \beta^{n}_{p+1,\mathrm{min}(P_{s})}.
\end{equation}
These set of identities can be easily inverted, giving  
\begin{equation}
\label{eq:mulinv}
     \beta^{n}_{p,p'} = \sum_{p'' \succcurlyeq p} \Big( \mu^{n}_{p''} + \sum_{p''' \prec p'} \mu^{n}_{p'',p'''} \Big),     
\end{equation}
for $n \in \ZZ$ and $p, p' \in P$ such that $p \succcurlyeq p'$. 
The analogous results hold if we had assumed that the filtrations were continuous on the right (and the other corresponding hypotheses) instead. 

%%%%%%%%%%%%%%%%%%%%%%%%%%%%%%%%%%%%%%%%%%%%%%%%%%%%%%%%%%%%%%%%%%%%%%%%%%%%%%%%%%%%%%%%%%%%%%%%%%%%%%%%%%%%%%%%%%%%%%%%%%%%%%%%%%%%%%%%%%%%%%%%
\subsubsection{\texorpdfstring{Relationship between persistent homology and cohomology}{Relationship between persistent homology and cohomology}} 
\label{subsubsection:homocohomo}

The persistent homology groups of an increasingly filtered homological dg module and the associated barcodes are defined \emph{mutatis mutandi}. 
The homological situation naturally appears in the literature for a topological space $X$ provided with an increasing filtration $\{ X_{p} \}_{p \in P}$ of topological subspaces. 
In that case the filtered homological dg module is the complex of singular chain $C_{\bullet}(X,k)$ provided with an increasing filtration $\{ F_{p}C_{\bullet}(X,k) \}_{p \in P}$, 
where $F_{p}C_{\bullet}(X,k) = C_{\bullet}(X_{p},k)$. 
We could also consider instead a simplicial set $\mathcal{X}$, provided with an increasing filtration $\{ \mathcal{X}_{p} \}_{p \in P}$ of simplicial sets, 
and we regard thus the complex $C_{\bullet}(\mathcal{X},k)$ of normalized chains computing the simplicial homology, provided with an increasing filtration $\{ F_{p}C_{\bullet}(\mathcal{X},k) \}_{p \in P}$, 
where $F_{p}C_{\bullet}(\mathcal{X},k) = C_{\bullet}(\mathcal{X}_{p},k)$ (see \cite{Wei94}, Ch. 8).  
The corresponding cohomological complexes were mentioned in the paragraph after Remark \ref{remark:hoco}. 

Now, let $C = C_{\bullet}$ be a homological dg module provided with an increasing $P$-filtration. 
Its graded dual $C^{\#}$ is a cohomological dg module with $(C^{\#})^{n} = C_{n}^{*}$, 
for all $n \in \ZZ$, where $(\place)^{*}$ denotes the usual dual functor.
It also has a decreasing $P$-filtration $\{ F^{p}(C^{\#}) \}_{p \in P}$, 
where $F^{p}(C^{\#})$ is the kernel of the map $C^{\#} \rightarrow (F_{p}C)^{\#}$ given by the graded dual of the canonical inclusion map $F_{p}C \rightarrow C$, for $p \in P$. 
This in particular applies to the situation we recalled in the previous paragraph, since $C^{\bullet}(X,k) = (C_{\bullet}(X,k))^{\#}$ or $C^{\bullet}(\mathcal{X},k) = (C_{\bullet}(\mathcal{X},k))^{\#}$ follows the same patterns. 
Given $p' \preccurlyeq p$ we further have the commutative diagram
\[
\xymatrix
{
0
\ar[r]
&
F^{p}C^{\#}
\ar[r]
\ar[d]^{i^{p,p'}}
&
C^{\#}
\ar[r]
\ar@{=}[d]
&
(F_{p}C)^{\#}
\ar[r]
\ar[d]^{i_{p',p}^{\#}}
&
0
\\
0
\ar[r]
&
F^{p'}C^{\#}
\ar[r]
&
C^{\#}
\ar[r]
&
(F_{p'}C)^{\#}
\ar[r]
&
0
}
\]
where $i^{p,p'} : F^{p}C^{\#} \rightarrow F^{p'}C^{\#}$ is the canonical inclusion induced (precisely to make the previous diagram commute) by the canonical inclusion 
$i_{p',p} : F_{p'}C \rightarrow F_{p}C$. 
Note that if the filtration on $C$ is continuous on the left (or right), the same happens for $C^{\#}$.
By taking cohomology and using that it commutes with taking graded duals, since the latter are exact functors, we get 
\[
\xymatrix@C-10pt
{
\dots
\ar[r]
&
H^{n}(F^{p}C^{\#})
\ar[r]
\ar[d]^{H^{n}(i^{p,p'})}
&
H^{n}(C^{\#})
\ar[r]
\ar@{=}[d]
&
(H_{n}(F_{p}C))^{*}
\ar[r]
\ar[d]^{H_{n}(i_{p',p})^{*}}
&
H^{n+1}(F^{p}C^{\#})
\ar[r]
\ar[d]^{H^{n+1}(i^{p,p'})}
&
\dots
\\
\dots
\ar[r]
&
H^{n}(F^{p'}C^{\#})
\ar[r]
&
H^{n}(C^{\#})
\ar[r]
&
(H_{n}(F_{p'}C))^{*}
\ar[r]
&
H^{n+1}(F^{p'}C^{\#})
\ar[r]
&
\dots
}
\]
This in particular tells us that we have a long exact sequence of (left) persistence modules 
\[     \dots \rightarrow H^{n}(\mathcal{D}) \overset{f^{n}}{\rightarrow} H^{n}(C^{\#}) \otimes k_{P} \overset{g^{n}}{\rightarrow} (H_{n}(\mathcal{D}'))^{\#} \overset{h^{n}}{\rightarrow} H^{n+1}(\mathcal{D}) \overset{f^{n+1}}{\rightarrow} \dots     \]
where $H^{n}(C^{\#}) \otimes k_{P}$ is the persistence module with the action coming from the usual multiplication on the second factor, 
and $\mathcal{D}' = \oplus_{p \in P} F_{p}C$ is the corresponding space 
whose $n$-th homology group gives the $n$-th persistent homology of $C_{\bullet}$. 

Suppose now that $H_{n}(F_{p}C)$ (or $H^{n}(F^{p}C^{\#})$) is finite dimensional, for all $n \in \ZZ$ and $p \in P$, 
and that the previous filtration of $H_{n}(C)$ (or $H^{n}(C^{\#})$) is bounded, for all $n \in \ZZ$.  
The locally finite dimensional hypothesis implies that the persistence modules considered before are pointwise finite dimensional. 
In particular, we may split 
$H^{n}(\mathcal{D})$ as a direct sum 
\[     H^{n}(\mathcal{D}) \simeq \bigoplus_{a \in A^{n}} k_{I_{a}} \oplus \bigoplus_{b \in B^{n}} k_{I_{b}},     \]
where the set of indices $A^{n}$ contains all intervals that are unbounded below, and $B^{n}$ all intervals that are bounded below. 
On the one hand, by Fact \ref{fact:1} we see that the restriction of $f^{n}$ to any direct summand of $H^{n}(\mathcal{D})$ isomorphic to $k_{I}$
such that $I$ is bounded below is zero. 
On the other hand, the fact that the filtration on $C^{\bullet}$ is exhaustive tells us that the restriction of $f^{n}$ to 
$K = \oplus_{a \in A^{n}} k_{I_{a}}$ is injective, for the persistence module $H_{n}((\mathcal{D}')^{\#})$ evaluated at sufficiently low $p$ 
should vanish. 
Indeed, let us assume that $m \in K(p)$, for some $p \in P$, satisfies that $f^{n}(p)(m) = 0$. 
Since $m \in K(p)$, there exists a finite set of indices $\bar{A} \subseteq A^{n}$ such that $m = \sum_{a \in \bar{A}} m_{a}$, 
where $m_{a} \in k_{I_{a}}(p)$ is nonzero. 
Taking $p' \preccurlyeq p$ sufficiently low, so that $H^{n}(F^{p'}(C^{\#})) = H^{n}(C^{\#})$, we get that $f^{n}|_{K}(p')$ is injective. 
Since $\omega^{p'}_{p}.m$ is in its kernel, we get that $\omega^{p'}_{p}.m = 0$. 
However, taking into account that $K$ has no torsion, for each $I_{a}$ is unbounded below, we conclude that $m = 0$. 

Consider now the decomposition 
\[     (H_{n}(\mathcal{D}'))^{\#} \simeq \bigoplus_{a \in A'_{n}} k_{J_{a}} \oplus \bigoplus_{b \in B'_{n}} k_{J_{b}},     \]
where the set of indices $A'_{n}$ contains all intervals that are unbounded above, and $B'_{n}$ all intervals that are bounded above. 
The Fact \ref{fact:1} tells us that the image of $g^{n}$ lies in the direct summand $K' = \bigoplus_{a \in A'_{n}} k_{J_{a}}$. 
Moreover, taking into account that the filtration on $H^{\bullet}(C^{\#})$ is bounded, we see that $K'$ 
lies in the image of $g^{n}$. 
Indeed, take $m \in k_{J_{a}}(p)$, for some $a \in A'_{n}$ and $p \in J_{a}$, and 
define $(g')^{n}$ to be the composition of $g^{n}$ with the canonical projection from $(H_{n}(\mathcal{D}'))^{\#}$ onto 
$K'$. 
Choosing if necessary $p' \succcurlyeq p$ large enough, so that $H^{n}((F_{p'}C)^{\#})) = H^{n}(C^{\#})$, we see that 
$(g')^{n}(p')$ is surjective. 
By the structure of the module $K'$ there exists $m' \in K'(p')$ such that $\omega^{p}_{p'}.m' = m$. 
The surjectivity of $(g')^{n}(p')$ tells us that there exists $n' \in H^{n}(C^{\#}) \otimes k_{P}(p')$ such that $g^{n}(p')(n') = m'$, 
which in turn implies that $g^{n}(p)(\omega^{p}_{p'}.n') = \omega^{p}_{p'}.g^{n}(p')(n') = \omega^{p}_{p'}.m' = m$. 

As a consequence of the two previous paragraphs we get that $h^{n}$ induces an isomorphism 
\[     \bigoplus_{b \in B^{n}} k_{I_{b}} \simeq \bigoplus_{b \in B'_{n}} k_{J_{b}}.     \]
By the Krull-Schmidt-Azumaya theorem (see \cite{AF92}, Thm. 12.6), these decompositions are equivalent. 
In particular, since there are no unbounded below intervals on the right and no unbounded above on the left, 
there are only bounded intervals in the previous decompositions. 
Furthermore, we also obtained a short exact sequence 
\begin{equation}
\label{eq:seq}
0 \rightarrow \bigoplus_{a \in A^{n}} k_{I_{a}} \rightarrow H^{n}(C^{\#}) \otimes k_{P} \rightarrow \bigoplus_{a \in A'_{n}} k_{J_{a}} \rightarrow 0     
\end{equation}     
of persistence modules. 
We note the easy fact that a decomposition of a pointwise persistence module $M$ with indecomposables given by interval modules, where the intervals are only unbounded below (resp., above) is completely determined by the dimension function of $M$. 
Moreover, since the dimension function is \emph{additive} (sometimes called an \emph{Euler-Poincar\'e map}), \textit{i.e.} for any short exact sequence of persistence modules of the form 
\[     0 \rightarrow M' \rightarrow M \rightarrow M'' \rightarrow 0,     \]
we have that $\dim M(p) = \dim M'(p) + \dim M''(p)$, 
for all $p \in P$, we see from \eqref{eq:seq} that there is a bijection between the family of intervals appearing in the 
decomposition of the leftmost term and the family of intervals appearing in the rightmost one given by $I \mapsto P \setminus I$. 
We have thus obtained a (fairly) simple proof of the following result. 
%%%%%%%
\begin{proposition}
\label{proposition:dual}
Let $C$ be a homological filtered dg module of locally finite dimensional homology, such that the induced filtration 
on each homology group $H_{n}(C)$ is bounded.  
Then, there are no unbounded below (resp., above) intervals of nonzero multiplicity in the barcode of persistent (resp., co)homology, 
the restriction of the barcode of $(n+1)$-th type for the persistent cohomology of $C^{\#}$ to the subset of 
bounded intervals coincides with the restriction of the barcode of $n$-th type for the persistent homology of $C$ to the same subset, 
and there is an equivalence between the restriction of the barcode of $n$-th type for the persistent cohomology of $C^{\#}$ to the subset of 
unbounded below intervals different from $P$ and the restriction of the barcode of $n$-th type for the persistent homology of $C$ to the subset of 
unbounded above intervals different from $P$ given by sending a bounded below interval $I$ to $P \setminus I$.  
\end{proposition}
%%%%%%%
This property has been observed for finite dimensional complexes $C$ (given as those computing cellular homology of a finite CW complex), 
provided with the filtration induced by the cellular decomposition, in \cite{dSMVJ11bis}, Prop. 2.3 and 2.4. 
Our proof is more general, and in our opinion much more direct. 

%%%%%%%%%%%%%%%%%%%%%%%%%%%%%%%%%%%%%%%%%%%%%%%%%%%%%%%%%%%%%%%%%%%%%%%%%%%%%%%%%%%%%%%%%%%%%%%%%%%%%%%%%%%%%%%%%%%%%%%%%%%%%%%%%%%%%%%%%%%%%%%%
\subsubsection{\texorpdfstring{Distances between barcodes}{Distances between barcodes}} 
\label{subsubsection:dist}

The story about persistent cohomology does not quite end at this point, but it barely begins, because one is further interested in comparing two structure of barcodes of $n$-th type 
coming from two different filtered dg modules $(C,d_{C})$ and $(C',d_{C'})$, each of which is filtered over $\RR$ (with the induced order). 
We are going to consider more generally filtrations over any totally ordered set $P$ such that the order topology is separable, 
provided with a finite metric such that the order topology coincides with that induced by the metric, 
and the metric satisfies the \emph{ternary (betweenness) condition} given by $d_{P}(p,p'') = d_{P}(p,p') + d_{P}(p',p'')$, 
for all $p, p', p'' \in P$ satisfying that $p \preccurlyeq p' \preccurlyeq p''$. 
As explained in the paragraph previous to Remark \ref{remark:vr}, one deals in general with the case of a topological space or the simplicial set embedded in $\RR^{n}$, thus metric spaces, 
and the filtration is constructed by some more or less established methods using the metric (using the Vietoris-Rips complex, or the \v{C}ech complex, etc.). 
Also, as explained in Subsubsection \ref{subsubsection:homocohomo}, 
if $X$ denotes a topological space embedded in $\RR^{n}$ and $C = C_{\bullet}$ is the (homological) dg module associated to $X$ using a particular method of topological data analysis, 
we will consider the cohomological cochain complex $C^{\#}$
together with the filtration $F^{p}(C^{\#})$ given by the kernel of the map 
$C^{\#} \rightarrow (G_{p}C)^{\#}$, where $G_{\epsilon}C$ 
denotes the usual $\RR$-filtration of the corresponding complex considered in usual topological data analysis 
(see Remark \ref{remark:vr} for the case of the Vietoris-Rips complex).  
Furthermore, in the majority of the cases this filtration will be in fact essentially finite and continuous on the left 
(\textit{e.g.}, the Vietoris-Rips filtration $\operatorname{VR}_{\epsilon}(X)$ of a finite metric space $X$, 
see Remark \ref{remark:vr}). 

Now, keeping in mind the same examples for the filtered dg modules we explained at the previous paragraph, 
the general principle of comparison is (by paraphrasing S. Weinberger in \cite{We11}, p. 38), roughly speaking, that two filtered dg modules 
should be ``similar'' if, for each $n \in \ZZ$, the multiplicities of the barcodes of $n$-th type are ``similar''. 
This comparison may be achieved by different methods, one of the most typical is by endowing the collection of barcodes with a \emph{pseudometric} 
(or \emph{semimetric}). 
We shall use the word \emph{distance} or even \emph{metric} on the corresponding set $S$, 
instead of pseudometric or semimetric, and we may allow it to be infinite, \textit{i.e.} a distance is 
a map $d : S \times S \rightarrow \RR_{\geq 0} \sqcup \{ + \hskip -0.4mm \infty \}$ satisfying that $d(s,s') = d(s',s)$, and 
$d(s,s'') \leq d(s,s') + d(s',s'')$, for all $s, s ', s'' \in S$, where we understand that $+ \hskip -0.4mm \infty + r = 
r + (+ \hskip -0.4mm \infty) = + \hskip -0.4mm \infty$, for all $r \in \RR_{\geq 0} \sqcup \{ + \hskip -0.4mm \infty \}$ 
(see \cite{BBI01}, Def. 1.1.1 and 1.1.4. However, we warn the reader that this terminology is different from the one used by other authors). 
If the image of the metric is included in $\RR_{\geq 0}$, we say it is \emph{finite}, and if it further 
satisfies that $d(s,s') = 0$ implies $s = s'$, we will say it is \emph{strict}.
Again, there may be different choices: we shall however deal with one, called the \emph{bottleneck (or matching) distance} 
(see \cite{DD13}, p. 54, p. 360). 
Our exposition more or less follows that in \cite{BGMP14}, Def. 2.7, with some minor variations. 

Assume for this subsubsection that $P$ is not only a totally ordered set such that the order topology is separable, 
but it is also provided with a finite metric $d_{P}$ such that the order topology coincides with the one induced by the metric, 
and the metric satisfies the ternary betweenness property given by $d_{P}(p,p'') = d_{P}(p,p') + d_{P}(p',p'')$, 
for all $p, p', p'' \in P$ satisfying that $p \preccurlyeq p' \preccurlyeq p''$.  
Since the order topology is easily seen to be Hausdorff, we see that $d_{P}$ is a strict metric. 
We recall that $\mathcal{I}_{P}$ is the set of all intervals of $P$. 
Given $I \in \mathcal{P}$, we denote by $\ell(I)$ the supremum of the set 
\[     \{ d_{P}(p',p) : \text{$p', p \in I$ such that $p' \preccurlyeq p$} \}.     \]
More generally, if $S \subseteq P$ is nonempty, define $\ell(S)$ as the supremum of the set 
\[     \{ \ell(I) : \text{$I \in \mathcal{I}_{P}$ and $I \subseteq S$}\}.     \]
Since any point is an interval, the previous set is nonempty. 
Finally, we define $\ell(\emptyset) = 0$. 

Consider now the set $\mathcal{I}_{P} \sqcup \{ * \}$ and define a metric $\tilde{d}$ on it as follows. 
For two intervals $I$ and $J$ in $P$, take $K$ the intersection of all intervals containing $I$ and $J$. 
This is clearly and interval of $P$. 
Define
\[     \tilde{d}(I,J) = \operatorname{max}(\ell(K \setminus I), \ell(K \setminus J)),    \]
together with
\[     \tilde{d}(I, *)  = \tilde{d}(*, I) = \ell(I)/2,  \hskip 0.5cm \text{ and }\hskip 0.5cm \tilde{d}(*,*) = 0,   \]
for all $I$ in $\mathcal{I}_{P}$. 
It is straightforward but tedious to prove that this satisfies the axioms of a distance function.
Given now two barcodes $\mathcal{S} = (S,v)$ and $\mathcal{S}' = (S,v')$ on $P$
the bottleneck distance between them 
in $\operatorname{Multi}(\mathcal{I}_{P})$ is defined as
\[     d_{\mathcal{B}}(\mathcal{S},\mathcal{S}) = \underset{\mathcal{P} \in \operatorname{PM}(\mathcal{S},\mathcal{S}')}{\operatorname{inf}} 
\underset{(\hat{s}_{1},\hat{s}_{2}) \in \mathcal{|P|}}{\operatorname{sup}} \tilde{d}(\hat{s}_{1},\hat{s}_{2}),     \]
where $\hat{s}_{1}$ and $\hat{s}_{2}$ are elements of $\mathcal{I}_{P} \sqcup \{ * \}$, 
and we recall that $\operatorname{PM}(\mathcal{S},\mathcal{S}')$ denotes the set of partial matchings between $\mathcal{S}$ and $\mathcal{S}'$ recalled at the first paragraph of Subsubsection \ref{subsubsection:bar}. 
The proof that the previous expression is indeed a metric is standard, but it can also be deduced from our Proposition \ref{proposition:metric}. 

Let us now assume that our dg modules $C$ and $C'$ have an essentially discrete $P$-filtration 
which is continuous on the left. 
We shall spell out more explicitly what is the bottleneck metric between the barcodes of $n$-th type associated to them. 
Moreover, this alternative description will be useful in the sequel. 
We also suppose that their cohomologies $H^{\bullet}(\mathcal{D})$ and $H^{\bullet}(\mathcal{D}')$ of the respective Rees dg algebras 
are locally finite dimensional. 
Using the decomposition (given by Theorem \ref{theorem:cb}) of the $n$-th cohomology $H^{n}(\mathcal{D})$ of the Rees construction \eqref{eq:reescomplex} into interval persistence modules 
of the form described in Remark \ref{remark:essfin}. 
Let $P_{s}$ and $P'_{s}$ be subposets of $P$ satisfying the hypothesis of the previously mentioned remark. 
For each $n \in \ZZ$ and each interval of $P_{s}$, which we will take without loss of generality to be of the form 
$R_{p_{1},p_{2}} = \{ p \in P : p_{1} \succcurlyeq p \succ p_{2} \}$, for some $p_{1} \in P_{s}$ and $p_{2} \in P_{s} \sqcup \{- \infty\}$ 
such that $p_{1} \succ p_{2}$, we choose a set of cocycles 
\[     \{ z^{n,m}_{p_{1},p_{2}} : m \in \NN_{\leq \mu^{n}_{R_{p_{1},p_{2}}}} \} \subseteq F^{p_{1}}C^{n}     \] 
of $\mathcal{D}$, satisfying that the image of the $k$-submodule expanded by the set of their cohomology classes under the canonical projection 
\[     H^{n}(F^{p_{1}}C) \rightarrow \frac{H^{n}(F^{p_{1}}C)}{\mathrm{Im}(H^{n}(F^{p_{1}+1}C) \rightarrow H^{n}(F^{p_{1}}C))}     \] 
has dimension $\mu^{n}_{p_{1},p_{2}} = \mu^{n}_{R_{p_{1},p_{2}}}$, and that $\omega^{p_{2}}_{p_{1}}.z^{n,m}_{p_{1},p_{2}}$ is a coboundary, for every $m = 1, \dots, \mu^{n}_{p_{1},p_{2}}$. 
We recall that if $p \in P_{s}$, then $p + 1$ denotes its immediate successor in $P_{s}  \sqcup \{ + \hskip -0.4mm \infty\}$.
Let us call $\mathcal{B}^{n}$ the set formed by all $z^{n,m}_{p_{1},p_{2}}$ for 
$p_{1} \in P_{s}$, $p_{2} \in P_{s} \sqcup \{ - \infty \}$ such that $p_{1} \succ p_{2}$, and all positive integers $m \leq \mu^{n}_{p_{1},p_{2}}$.
The same applies to the other filtered dg module $(C',d_{C'})$. 
For every $n \in \ZZ$ and every interval $R'_{p'_{1},p'_{2}}$ defined analogously, 
where $p'_{1} \in P'_{s}$ and $p'_{2} \in P'_{s} \sqcup \{ - \infty \}$ satisfies that $p'_{1} \succ p'_{2}$, we shall denote the corresponding set of cocycles by
\[     \{ (z')^{n,m'}_{p'_{1},p'_{2}} : m' \in \NN_{\leq (\mu')^{n}_{R'_{p'_{1},p'_{2}}}} \} \subseteq F^{p'_{1}}(C')^{n},     \] 
and the union of these sets for all $p'_{1} \in P'_{s}$, $p'_{2} \in P'_{s} \sqcup \{ - \infty \}$ such that $p'_{1} \succ p'_{2}$ and all positive integers $m' \leq (\mu')^{n}_{R'_{p'_{1},p'_{2}}} = (\mu')^{n}_{p'_{1},p'_{2}}$ will be denoted by $(\mathcal{B}')^{n}$.

For fixed $n \in \ZZ$, we define the distance $d^{n}(z^{n,m}_{p_{1},p_{2}},z^{n,m'}_{p'_{1},p'_{2}}) = \tilde{d}(R_{p_{1},p_{2}},R'_{p'_{1},p'_{2}})$, for all $p_{1} \in P_{s}$, $p_{2} \in P_{s} \sqcup \{ - \infty \}$ such that $p_{1} \succ p_{2}$, $p'_{1} \in P'_{s}$, $p'_{2} \in P'_{s} \sqcup \{ - \infty \}$ such that $p'_{1} \succ p'_{2}$, 
and all positive integers $m \leq \mu^{n}_{p_{1},p_{2}}$ and $m' \leq (\mu')^{n}_{p'_{1},p'_{2}}$. 
We now define the distance between $\mathcal{B}^{n}$ and $(\mathcal{B}')^{n}$ as follows. 
Regarding the latter two as multisets, then the set $\operatorname{PM}(\mathcal{B}^{n}, (\mathcal{B}')^{n})$ of partial matchings coincides 
with the collection of subsets $\mathcal{P}$ of 
\[     \Big((\mathcal{B}^{n} \sqcup \{ * \}) \times ((\mathcal{B}')^{n} \sqcup \{ * \})\Big) \setminus \{ (*,*) \}     \] 
satisfying that, for any $z \in \mathcal{B}^{n}$ (resp., $z' \in (\mathcal{B}')^{n}$), the cardinality of the set given by $(\{ z \} \times ((\mathcal{B}')^{n} \sqcup \{ * \})) \cap \mathcal{P}$ 
(resp., $((\mathcal{B}^{n} \sqcup \{ * \}) \times \{ z' \}) \cap \mathcal{P}$) is exactly one. 
We also define the distances $d^{n}(z^{n,m}_{p_{1},p_{2}},*) = \tilde{d}(R_{p_{1},p_{2}},*)$ and $d^{n}(*,z^{n,m'}_{p'_{1},p'_{2}}) = \tilde{d}(*,R'_{p'_{1},p'_{2}})$. 
Then, the bottleneck metric between the barcodes $n$-th type associated to $C$ and $C'$ coincides with 
\[     \underset{\mathcal{P} \in \operatorname{PM}(\mathcal{B}^{n}, (\mathcal{B}')^{n})}{\operatorname{inf}} \underset{(z,z') \in \mathcal{P}}{\operatorname{sup}} d^{n}(z,z').     \]
We shall denote this number by $\partial_{\mathcal{B}}^{n}(H^{\bullet}(\mathcal{D}),H^{\bullet}(\mathcal{D}'))$. 
Furthermore, we set 
\begin{equation}
\label{eq:bm}
     \partial_{\mathcal{B}}(H^{\bullet}(\mathcal{D}),H^{\bullet}(\mathcal{D}')) = \operatorname{sup}_{n \in \ZZ} \partial_{\mathcal{B}}^{n}(H^{\bullet}(\mathcal{D}),H^{\bullet}(\mathcal{D}')),
\end{equation}
and call it the \emph{bottleneck distance between the (generalized barcodes associated to) filtered dg modules $C$ and $C'$}. 
We remark that it obviously satisfies the properties of a distance function. 
Roughly speaking, for computing the bottleneck distance between the barcodes of $n$-th type associated to $C$ and $C'$, 
one considers all the manners of (partially) pairing the sets of cocycles $\mathcal{B}^{n}$ and $(\mathcal{B}')^{n}$ and for such partial pairing one takes the maximum of the distances between the paired elements. 
The value of the bottleneck metric is just the infimum of all the numbers obtained before, when considering all the possible partial pairings. 

%%%%%%%%%%%%%%%%%%%%%%%%%%%%%%%%%%%%%%%%%%%%%%%%%%%%%%%%%%%%%%%%%%%%%%%%%%%%%%%%%%%%%%%%%%%%%%%%%%%%%%%%%%%%%%%%%%%%%%%%%%%%%%%%%%%%%%%%%%%%%%%%
\subsection{\texorpdfstring{$A_{\infty}$-algebras and Kadeishvili's theorem}{A-infinity-algebras and Kadeishvili's theorem}} 
\label{subsection:kad}

The notion of $A_{\infty}$-algebra was introduced by J. Stasheff in \cite{Sta63} in his study of homotopy theory of loop spaces. 
We refer the reader to \cite{Prou11}, Chapitre 3, or \cite{LH03}, Chapitre 1, for standard references, even though we follow the sign and grading conventions explained in detail in \cite{Her14}, Subsection 2.2. 

From now on we shall assume that $(P , \succcurlyeq)$ is a poset provided with an associative and commutative product structure $\star : P \times P \rightarrow P$ 
satisfying the \emph{monotonicity} property, \textit{i.e.} if $p_{1} \succcurlyeq p'_{1}$ and $p_{2} \succcurlyeq p'_{2}$, for $p_{1}, p_{2}, p'_{1}, p'_{2} \in P$, then $p_{1} \star p_{2} \succcurlyeq p'_{1} \star p'_{2}$. 
Let $A$ be a graded module provided with an extra grading indexed by $P$ (called the \emph{Adams degree}), \textit{i.e.} $A = \oplus_{p \in P} \oplus_{n \in \ZZ} A^{p,n}$. 
We remark that our definition of Adams degree is slightly more general than the one appearing in \cite{LPWZ09}, since $P$ is not necessarily a group as in that article. 
We shall provide the tensor product $A^{\otimes i}$ with the grading over $P \times \ZZ$ given by setting $A^{p_{1}, n_{1}} \otimes \dots \otimes A^{p_{i},n_{i}}$ in cohomological degree $n_{1} + \dots + n_{i}$ 
and Adams degree $p_{1} \star \dots \star p_{i}$. 
An \emph{(Adams graded) $A_{\infty}$-algebra} structure on $A$ is a collection of maps 
$m_{i} : A^{\otimes i} \rightarrow A$ for $i \in \NN$ of cohomological degree $2-i$ and preserving the Adams degree satisfying the \emph{Stasheff identities} $\mathrm{SI}(n)$ given by   
\begin{equation}
\label{eq:ainftyalgebra}
   \sum_{(r,s,t) \in \mathscr{I}_{n}} (-1)^{r + s t}  m_{r + 1 + t} \circ (\mathrm{id}_{A}^{\otimes r} \otimes m_{s} \otimes \mathrm{id}_{A}^{\otimes t}) = 0,
\end{equation} 
for $n \in \NN$, where $\mathscr{I}_{n} = \{ (r,s,t) \in \NN_{0} \times \NN \times \NN_{0} : r + s + t = n \}$. 
Given $N \in \NN$, if $A$ is provided only with the morphisms $m_{i}$ for $i \in \NN_{\leq N}$ and satisfy the Stasheff identities $\mathrm{SI}(n)$ for $n \in \NN_{\leq N}$, we say that it is an \emph{$A_{N}$-algebra}.
Note that the first Stasheff identity $\mathrm{SI}(1)$ means that $m_{1}$ is a differential of $A$, 
so we may consider the cohomology $k$-module $H^{\bullet}(A)$ given by the quotient $\mathrm{Ker}(m_{1})/\mathrm{Im}(m_{1})$. 

An $A_{\infty}$-algebra is called \emph{(strictly) unitary} if there is a map $\eta_{A} : k \rightarrow A$ of complete degree zero (supposing that $P$ has a unit for $\star$) such that $m_{i} \circ (\mathrm{id}_{A}^{\otimes r} \otimes \eta_{A} \otimes \mathrm{id}_{A}^{\otimes t})$
vanishes for all $i \neq 2$ and all $r, t \geq 0$ such that $r+1+t = i$. 
We shall usually denote the image of $1_{k}$ under $\eta_{A}$ by $1_{A}$, and call it the \emph{(strict) unit} of $A$.  
We say that a unitary or nonunitary $A_{\infty}$-algebra is called \emph{minimal} if $m_{1}$ vanishes. 
Note that all the previous definitions can be applied as well to $A_{N}$-algebras, for $N \in \NN$. 
We see that a (unitary) dg algebra $(A,d_{A},\mu_{A})$ is a particular case of (unitary) $A_{\infty}$-algebra, where $m_{1} = d_{A}$ is the differential and $m_{2} = \mu_{A}$ is the product. 

A \emph{morphism of $A_{\infty}$-algebras} $f_{\bullet} : A \rightarrow B$ between two $A_{\infty}$-algebras $A$ and $B$ is a collection of morphisms 
of the underlying graded $k$-modules $f_{i} : A^{\otimes n} \rightarrow B$ of cohomological degree $1-i$ for $i \in \NN$ and preserving the Adams degree
satisfying the \emph{Stasheff identities on morphisms} $\mathrm{MI}(n)$ given by  
\begin{equation}
\label{eq:ainftyalgebramor}
   \sum_{(r,s,t) \in \mathscr{I}_{n}} (-1)^{r + s t}  f_{r + 1 + t} \circ (\mathrm{id}_{A}^{\otimes r} \otimes m_{s}^{A} \otimes \mathrm{id}_{A}^{\otimes t}) 
   = \sum_{q \in \NN} \sum_{\bar{i} \in \NN^{q, n}} (-1)^{w} m_{q}^{B} \circ (f_{i_{1}} \otimes \dots \otimes f_{i_{q}}),
\end{equation} 
for $n \in \NN$, where $w = \sum_{j=1}^{q} (q-j) (i_{j} - 1)$ and $\NN^{q,n}$ is the subset of $\NN^{q}$ of elements $\bar{i} = (i_{1},\dots,i_{q})$ such that $|\bar{i}| = i_{1} + \dots + i_{q} = n$. 
Given $N \in \NN$, if $A$ and $B$ are only $A_{N}$-algebras, a \emph{morphism of $A_{N}$-algebras} $f_{\bullet} : A \rightarrow B$ is a collection of morphisms 
of the underlying graded $k$-modules $f_{i} : A^{\otimes n} \rightarrow B$ of cohomological degree $1-i$ for $i \in \NN_{\leq N}$ and preserving the Adams degree 
satisfying the previous identities $\mathrm{MI}(n)$ for $n \in \NN_{\leq N}$.  
If $A$ and $B$ are unitary $A_{\infty}$-algebras, the morphism $f_{\bullet}$ is called \emph{(strictly) unitary} if $f_{1}(1_{A}) = 1_{B}$, and for all $i \geq 2$ we have that $f_{i}(a_{1}, \dots, a_{i})$ vanishes 
if there exists $j \in \{1, \dots, i \}$ such that $a_{j} = 1_{A}$. 
Notice that $f_{1}$ is a morphism of dg $k$-modules for the underlying structures on $A$ and $B$. 
We say that a morphism of (resp., unitary) $A_{\infty}$-algebras $f_{\bullet} : A \rightarrow B$ is a \emph{quasi-isomorphism} if $f_{1}$ is a quasi-isomorphism of the underlying complexes. 
We say that a morphism $f_{\bullet}$ is \emph{strict} if $f_{i}$ vanishes for $i \geq 2$. 
Note that all these definitions also apply to morphisms of $A_{N}$-algebras. 

The following result is well-known and due to T. Kadeishvili (see \cite{K80}, Thm. 1, for the case of dg algebras, or \cite{K82}, Thm., for the case of plain $A_{\infty}$-algebras). 
It includes however some extra conditions about the Adams degree which are useful for our study of persistent cohomology (the proof recalled in \cite{Her14}, Thm. 2.2, applies \textit{verbatim}). 
%%%%%%%
\begin{theorem}
\label{theorem:kadeish} 
Let $(A,m_{\bullet})$ be an (resp., a unitary) $A_{\infty}$-algebra and $f_{1} : H^{\bullet}(A) \rightarrow A$ be the composition of a section of the canonical projection $\Ker(m_{1}) \rightarrow H^{\bullet}(A)$ and the inclusion $\Ker(m_{1}) \subseteq A$ (resp., satisfying that $f_{1} \circ \eta_{H^{\bullet}(A)} = \eta_{A}$). 
Then there exists a structure of (resp., unitary) $A_{\infty}$-algebra on $H^{\bullet}(A)$ given by $\{ \bar{m}_{n} \}_{n \in \NN}$ and a quasi-isomorphism of (resp., unitary) $A_{\infty}$-algebras 
$f_{\bullet} : H^{\bullet}(A) \rightarrow A$ whose first component is $f_{1}$, such that $\bar{m}_{1} = 0$ and $\bar{m}_{2}$ is the multiplication induced by $m_{2}$. 
Moreover, all these possible structures of (resp., unitary) $A_{\infty}$-algebras on $H^{\bullet}(A)$ are quasi-isomorphic. 
Any of these quasi-isomorphic (resp., unitary) $A_{\infty}$-structures will be called a \emph{model}.  
\end{theorem}
%%%%%%%

%%%%%%%%%%%%%%%%%%%%%%%%%%%%%%%%%%%%%%%%%%%%%%%%%%%%%%%%%%%%%%%%%%%%%%%%%%%%%%%%%%%%%%%%%%%%%%%%%%%%%%%%%%%%%%%%%%%%%%%%%%%%%%%%%%%%%%%%%%%%%%%%%%%%%
\section{\texorpdfstring{$A_{\infty}$-algebra structure on persistent cohomology}{A-infinity-algebra structure on persistent cohomology}}
\label{section:a_infpcoho}

%%%%%%%%%%%%%%%%%%%%%%%%%%%%%%%%%%%%%%%%%%%%%%%%%%%%%%%%%%%%%%%%%%%%%%%%%%%%%%%%%%%%%%%%%%%%%%%%%%%%%%%%%%%%%%%%%%%%%%%%%%%%%%%%%%%%%%%%%%%%%%%%%%%%%
\subsection{Generalities}
\label{subsection:gen}

We recall that $(P , \succcurlyeq)$ is a poset provided with an associative and commutative product structure $\star : P \times P \rightarrow P$   
satisfying the monotonicity property. 
Suppose now that the filtered dg module $(C,d_{C})$ we have considered in Subsection \ref{subsection:pers} is in fact a dg algebra, and that the filtration is \emph{multiplicative}, \textit{i.e.} 
$F^{p}C \cdot F^{q}C \subseteq F^{p \star q}C$, for all $p, q \in P$. 
If $P$ has a unit $0$ for $\star$, and $(C,d_{C})$ has also a unit $1_{C}$, we further assume that $1_{C} \in F^{0}C$. 
In this case it is easily verified that \eqref{eq:reescomplex} is in fact an Adams graded (unitary) dg algebra, for the obvious product (and unit). 
In our situation of interest, \textit{e.g.} $C^{\bullet}(\mathcal{X},k)$ for some simplicial set $X$, they are unitary dg algebras for the cup product, 
where the unit is given the map $C_{\bullet}(\mathcal{X},k) \rightarrow k$ whose restriction to $C_{n}(\mathcal{X},k)$ vanishes for $n \neq 0$ and sends any $0$-simplex to $1$ 
(see \cite{Ha02}, Ch. 3, Sec. 2, or \cite{FHT01}, Ch. 10, (d)).
However, given any (increasing) filtration of the simplicial complex $\mathcal{X}$ indexed by (say) the integers (as it may be the case in algebraic topology), 
the induced filtration for $C^{\bullet}(\mathcal{X},k)$ need not be in general multiplicative if one uses the multiplication of the integers (see \cite{Sh76}). 
If we consider the case of the complex of normalized cochains $C^{\bullet}(\mathcal{X},k)$ of a simplicial set $\mathcal{X}$ which is provided with the \emph{skeletal filtration}, indexed by $p \in \ZZ$, 
(see \cite{FHT01}, Ch. 10, (a)), a simple computation using the explicit expression of the cup product of $C^{\bullet}(\mathcal{X},k)$ shows that the induced filtration $\{ F^{p}C(\mathcal{X},k) \}_{p \in \ZZ}$ 
is indeed multiplicative (\textit{cf.} \cite{FHT01}, Ch. 10, (d)). 
This example is nevertheless far from the motivations of topological data analysis, since in the latter situation the possible filtration of interest are somehow already prescribed, and constructed 
from the metric of the simplicial set $\mathcal{X}$ (which is induced from some embedding in $\RR^{n}$). 
For this reason we prefer to relax the hypotheses on the product of the set of indices $P$ (and even on the set itself), 
which also led us to consider more properly exact couples systems (than classical exact couples). 
For instance we may take the product in $P$ as the one given by $p \star p' = \operatorname{inf}(p,p')$, which is the coarsest possible in case $P$ is totally ordered and the members of the filtration 
are (not necessarily unitary) dg subalgebras of $(C,d_{C})$. 
These hypotheses are for instance always satisfied in the case of the filtration of a dg algebra of normalized cochains on a simplicial set provided with the induced filtration from any filtration 
of the simplicial set. 
In any case, from now on we shall assume that the filtration of the (unitary) dg algebras we are dealing with are multiplicative. 

Under the previous assumptions we see that \eqref{eq:reescomplex} is in fact a(n Adams graded unitary) dg algebra over $k$. 
Now, using Theorem \ref{theorem:kadeish}, we get that its cohomology is a (unitary) $A_{\infty}$-algebra over $k$, and they are in fact quasi-isomorphic, as (unitary) $A_{\infty}$-algebras over $k$. 
In other words, the first term $D = H^{\bullet}(\mathcal{D})$ of the exact couple system associated to $(C,d_{C})$ is naturally a (unitary) $A_{\infty}$-algebra over $k$ (this is slightly stronger, with respect to $D$, than the result stated in \cite{Ma13}, Example 4.3). 
Note that, since the multiplication of $D$ is induced from the multiplication of $\mathcal{D}$, the former is bilinear 
for the action of $k.\tilde{P}$. 

We would like to remark that the previous $A_{\infty}$-algebra is only a shadow of the complete structure of $A_{\infty}$-algebra over 
$k.\tilde{P}$ (defined on another space, not on $D$). 
Indeed, by the results of \cite{Her14}, Section 3 (adapted so that indices are more general than just the integers), 
one obtains (at least under some assumptions on the filtration) that there is an $A_{\infty}$-algebra over $k.\tilde{P}$ 
on a space $\mathscr{D}$, which is quasi-isomorphic as $A_{\infty}$-algebras over $k.\tilde{P}$ to $\mathcal{D}$, minimal in the sense of S. Sagave 
(see \cite{Sa10}), and also quasi-isomorphic (but only as $A_{\infty}$-algebras over $k$) to $D$.
In other words, the $A_{\infty}$-algebra structure of $D$ only deals with part of the whole picture. 
However, since $D$ seems to be a preferred choice in topological data analysis, and it already includes some higher homotopic information, 
we have decided to work with it. 
We believe nevertheless that the mentioned space $\mathscr{D}$ should also deserve some attention. 

%%%%%%%%%%%%%%%%%%%%%%%%%%%%%%%%%%%%%%%%%%%%%%%%%%%%%%%%%%%%%%%%%%%%%%%%%%%%%%%%%%%%%%%%%%%%%%%%%%%%%%%%%%%%%%%%%%%%%%%%%%%%%%%%%%%%%%%%%%%%%%%%%%%%%
\subsection{An extension of the bottleneck distance}
\label{subsection:bm}

We are now going to propose a finer but (essentially) equivalent (with respect to the bottleneck metric, if one forgets the extra structure) distance function 
in the space of $P$-filtered (unitary) dg algebras, where the filtration is continuous on the left and essentially discrete 
(under some multiplicativity assumptions). 
We suppose for the rest of the article $P$ is not only a totally ordered set such that the order topology is separable, 
but it is also provided with a strict metric $d_{P}$ such that the order topology coincides with the one induced by the metric, 
and the metric satisfies the ternary relation given by $d_{P}(p,p'') = d_{P}(p,p') + d_{P}(p',p'')$, 
for all $p, p', p'' \in P$ satisfying that $p \preccurlyeq p' \preccurlyeq p''$. 
Our idea is, \textit{grosso modo}, that the general principle of comparison recalled at the antepenultimate paragraph of Subsubsection \ref{subsubsection:dist} could be enriched as follows: two filtered dg modules 
$C$ and $C'$ (provided with further structure) are ``similar'' if not only the cohomology classes of $C$ and $C'$ which persist longer have similar multiplicities, but also the way the cohomology classes of larger persistence of $C$ interact 
is similar to the way the corresponding cohomology classes of $C'$ interact.
Since a typical manner in algebraic topology to encode how a set of cohomology classes (satisfying some conditions) interact is by means of the Massey products, and since 
they can be organised (getting rid of the restrictions mentioned previously) into the $A_{\infty}$-algebra structure on the corresponding cohomology (see \cite{LPWZ09}, Thm. 3.1 and Cor. A.5), 
we think that the most natural way to depict the previous interactions between cohomology classes is by taking into account the corresponding $A_{\infty}$-algebra structure. 
The definition of the extension of the bottleneck distance which considers these (higher) multiplications will require some preparations, 
which we now provide. 

Given $N \in \NN \sqcup \{ +\hskip -0.4mm \infty\}$, we will introduce the \emph{$A_{N}$-bottleneck metric} 
between two $P$-filtered dg algebras $A$ and $A'$, where the filtration is continuous on the left and essentially discrete. 
Let $P_{s}$ be a discrete subset of $P$ satisfying the condition of Remark \ref{remark:essfin} for $A$, 
and analogously $P'_{s}$ for $A'$. 
Furthermore, we assume that the respective Rees dg algebras have locally finite dimensional cohomology. 
We also suppose that $P_{s}$ and $P'_{s}$ are provided with associative and commutative binary operations $\star$ and $\star'$. 
This hypothesis is always met if the members of the corresponding filtrations are dg subalgebras and the operations 
are given by taking the minimum (or, roughly speaking, any situation where the product on the set of indices comes from the filtration itself). 
Given intervals $I$ and $I'$ of $P$, $\mu^{n}_{I}$ and $(\mu')^{n}_{I'}$ will denote the corresponding multiplicities of the barcodes at the intervals $I$ of $A$ and $I'$ of $A'$, as explained in Subsubsection \ref{subsubsection:bar}. 

By the comments at the last paragraph of Subsubsection \ref{subsubsection:dist}, for each $n \in \ZZ$ and each interval, 
which we may take without loss of generality to be of the form 
$R_{p_{1},p_{2}} = \{ p \in P : p_{1} \succcurlyeq p \succ p_{2} \}$, for some $p_{1} \in P_{s}$ and $p_{2} \in P_{s} \sqcup \{ - \infty\}$ 
such that $p_{2} \succ p_{1}$, we may choose a set of cocycles 
\[     \{ z^{n,m}_{p_{1},p_{2}} : m \in \NN_{\leq \mu^{n}_{R_{p_{1},p_{2}}}} \} \subseteq F^{p_{1}}A^{n}     \] 
of $\mathcal{D}$, satisfying that the image of the $k$-submodule expanded by the set of their cohomology classes under the canonical projection 
\[     H^{n}(F^{p_{1}}A) \rightarrow \frac{H^{n}(F^{p_{1}}A)}{\mathrm{Im}(H^{n}(F^{p_{1}+1}A) \rightarrow H^{n}(F^{p_{1}}A))}     \] 
has dimension $\mu^{n}_{p_{1},p_{2}} = \mu^{n}_{R_{p_{1},p_{2}}}$, and that $\omega^{p_{2}}_{p_{1}}.z^{n,m}_{p_{1},p_{2}}$ is a coboundary, for every $m = 1, \dots, \mu^{n}_{p_{1},p_{2}}$. 
We recall that if $p \in P_{s}$, then $p + 1$ denotes its immediate successor in $P_{s}  \sqcup \{ + \hskip -0.4mm \infty\}$.
Let us call $\mathcal{B}^{n}$ the set formed by all $z^{n,m}_{p_{1},p_{2}}$ for 
$p_{1} \in P_{s}$, $p_{2} \in P_{s} \sqcup \{ - \infty \}$ such that $p_{1} \succ p_{2}$ and all positive integers $m \leq \mu^{n}_{p_{1},p_{2}}$.
For every $n \in \ZZ$ and every interval $R'_{p'_{1},p'_{2}}$ of $P'$ defined analogously, 
for $p'_{1} \in P'_{s}$, and $p'_{2} \in P'_{s} \sqcup \{ - \infty \}$ such that $p'_{1} \succ p'_{2}$, we will denote the corresponding set of cocycles by
\[     \{ (z')^{n,m'}_{p'_{1},p'_{2}} : m' \in \NN_{\leq (\mu')^{n}_{R'_{p'_{1},p'_{2}}}} \} \subseteq F^{p'_{1}}(A')^{n},     \] 
and the union of these sets for all $p'_{1} \in P'_{s}$, $p'_{2} \in P'_{s} \sqcup \{ - \infty \}$ such that $p'_{1} \succ p'_{2}$ and all positive integers $m' \leq (\mu')^{n}_{R_{p'_{1},p'_{2}}} = (\mu')^{n}_{p'_{1},p'_{2}}$ will be denoted by $(\mathcal{B}')^{n}$.
Moreover we set
\[     \mathcal{B} = \bigcup_{n \in \ZZ} \mathcal{B}^{n} \text{\hskip 1cm and \hskip 1cm} \mathcal{B}' = \bigcup_{n \in \ZZ} (\mathcal{B}')^{n}.     \]

Note the easy fact that the cohomology classes of the elements $\omega^{p_{3}}_{p_{1}}.z^{n,m}_{p_{1},p_{2}}$, for $p_{1}, p_{3} \in P_{s}$ and $p_{2} \in P_{s} \sqcup \{ - \infty \}$
such that $p_{1} \succcurlyeq p_{3} \succ p_{2}$, $n \in \ZZ$, and positive integer $m \in \mu^{n}_{p_{1},p_{2}}$, 
form a $k$-basis of the cohomology $H^{\bullet}(\mathcal{D})$ of the Rees dg algebra of $A$. 
We shall denote the basis formed by the previous elements with $n$ fixed by $\hat{\mathcal{B}}^{n}$, and 
\[     \hat{\mathcal{B}} = \bigcup_{n \in \ZZ} \hat{\mathcal{B}}^{n}.     \]
The same considerations hold for $A'$, for which we denote the corresponding bases by $(\hat{\mathcal{B}}')^{n}$ and $\hat{\mathcal{B}}'$.  
As in the usual case, for fixed $n \in \ZZ$, we have the distance $d^{n}(z^{n,m}_{p_{1},p_{2}},z^{n,m'}_{p'_{1},p'_{2}}) = \tilde{d}(R_{p_{1},p_{2}},R'_{p'_{1},p'_{2}})$, for all $p_{1} \in P_{s}$, $p_{2} \in P_{s} \sqcup \{ - \infty \}$ such that $p_{1} \succ p_{2}$, $p'_{1} \in P'_{s}$, $p'_{2} \in P'_{s} \sqcup \{ - \infty \}$ such that $p'_{1} \succ p'_{2}$, and 
all positive integers $m \leq \mu^{n}_{p_{1},p_{2}}$ and $m' \leq (\mu')^{n}_{p'_{1},p'_{2}}$. 
We even extend this distance to define 
\[     d^{n}(\omega^{p_{3}}_{p_{1}}.z^{n,m}_{p_{1},p_{2}}, \omega^{p'_{3}}_{p'_{1}}.z^{n,m'}_{p'_{1},p'_{2}}) = \tilde{d}(R_{p_{3},p_{2}},R_{p'_{3},p'_{2}}),     \] 
for all $p_{1}, p_{3} \in P_{s}$ and $p_{2} \in P_{s} \sqcup \{ - \infty \}$ such that $p_{1} \succcurlyeq p_{3} \succ p_{2}$, 
$p'_{1}, p'_{3} \in P'_{s}$ and $p'_{2} \in P'_{s} \sqcup \{ - \infty \}$ such that $p'_{1} \succcurlyeq p'_{3} \succ p'_{2}$, and 
all positive integers $m \leq \mu^{n}_{p_{1},p_{2}}$ and $m' \leq (\mu')^{n}_{p'_{1},p'_{2}}$. 

Let $S$ and $S'$ be two nonempty subsets of $\hat{\mathcal{B}}^{n}$ and of $(\hat{\mathcal{B}}')^{n}$, respectively.
We recall that, considering the latter two as multisets, then the set $\operatorname{PM}(S, S')$ of partial matchings coincides 
with the collection of subsets $\mathcal{P}$ of 
\[     \Big((S \sqcup \{ * \}) \times (S' \sqcup \{ * \})\Big) \setminus \{ (*,*) \}     \] 
satisfying that, for any $z \in S$ (resp., $z' \in S'$), the cardinality of the set given by $(\{ z \} \times (S' \sqcup \{ * \})) \cap \mathcal{P}$ (resp., $((S \sqcup \{ * \}) \times \{ z' \}) \cap \mathcal{P}$) is exactly one. 
We also define the distances $d^{n}(\omega^{p_{3}}_{p_{1}}.z^{n,m}_{p_{1},p_{2}},*) = \tilde{d}(R_{p_{3},p_{2}},*)$ and $d^{n}(*,\omega^{p'_{3}}_{p'_{1}}.z^{n,m'}_{p'_{1},p'_{2}}) = \tilde{d}(*,R'_{p'_{3},p'_{2}})$, 
for all $p_{1} \in P_{s}$, $p_{2} \in P_{s} \sqcup \{ - \infty \}$ such that $p_{1} \succ p_{2}$, $p'_{1} \in P'_{s}$, $p'_{2} \in P'_{s} \sqcup \{ - \infty \}$ such that $p'_{1} \succ p'_{2}$, and 
all positive integers $m \leq \mu^{n}_{p_{1},p_{2}}$ and $m' \leq (\mu')^{n}_{p'_{1},p'_{2}}$. 
We now define the distance between $S$ and $S'$ by the formula 
\[     \hat{d}(S,S') = \underset{\mathcal{P} \in \operatorname{PM}(S,S')}{\operatorname{inf}} \underset{(z,z') \in \mathcal{P}}{\operatorname{sup}} d^{n}(z,z').     \]
Note that the previous distance depends in fact of the chosen bases $\hat{\mathcal{B}}^{n}$ and $(\hat{\mathcal{B}}')^{n}$, 
but we omitted them from the notation for simplicity. 
We extend the previous definition to consider the case that either $S$ or $S'$ coincides with $\{*\}$ by means of 
\[     \hat{d}(S,\{*\}) = \underset{z \in S}{\operatorname{sup}} \hskip 0.8mm d^{n}(z,*), \hskip 1cm \hat{d}(\{*\},S') = \underset{z' \in S'}{\operatorname{sup}} d^{n}(*,z'),    \]
and $\hat{d}(\{*\},\{*\}) = 0$. 
We have the following result.
%%%%%%%
\begin{fact}
\label{fact:2}
Let $n \in \ZZ$ and $A$, $A'$ and $A''$ be three $P$-filtered dg algebras, such that the filtrations are essentially discrete and continuous on the left. 
We suppose as usual the multiplicity property for the discrete subsets $P_{s}$, $P'_{s}$ and $P''_{s}$. 
Let $S$, $S'$, $S''$ be three elements of $(\mathscr{P}(\hat{\mathcal{B}}^{n}) \setminus \{ \emptyset \}) \sqcup \{ \{ * \} \}$, $(\mathscr{P}((\hat{\mathcal{B}}')^{n}) \setminus \{ \emptyset \}) \sqcup \{ \{ * \} \}$ and $(\mathscr{P}((\hat{\mathcal{B}}'')^{n})  \setminus \{ \emptyset \}) \sqcup \{ \{ * \} \}$, respectively, 
constructed as in the second previous paragraph. 
Then, we have the property 
\[     \hat{d}(S,S'') \leq \hat{d}(S,S') + \hat{d}(S',S'').     \]
\end{fact}
%%%%%%%
\begin{proof}
Define $d_{1} = \hat{d}(S,S')$ and $d_{2} = \hat{d}(S',S'')$. 
If they are infinite there is nothing to prove, so we assume they are finite. 
We shall make the proof for the case that $S, S', S''$ do not coincide with $\{ * \}$, for the latter is clear. 
Take any real number $\epsilon > 0$ and let 
$\mathcal{P}^{1} \in \operatorname{PM}(S,S')$ and $\mathcal{P}^{2} \in \operatorname{PM}(S',S'')$ be such that 
\[     \underset{(u,v) \in \mathcal{P}^{i}}{\operatorname{sup}} d^{n}(u,v) \leq d_{i} + \epsilon,     \]
for $i = 1, 2$. 
Define a partial matching $\mathcal{P}$ between $S$ and $S''$, which we call the \emph{composition} of $\mathcal{P}^{1}$ and $\mathcal{P}^{2}$, 
as follows.  
Let $S'_{0} \subset S'$ be the subset of elements $z' \in S'$ such that $(*,z') \notin \mathcal{P}^{1}$ and $(z',*) \notin \mathcal{P}^{2}$. 
Define $S_{0} \subset S$ (resp., $S_{0}'' \subset S''$) as the set of elements $z \in S$ (resp., $z'' \in S''$) 
such that there exists $z' \in S'_{0}$ satisfying that $(z,z') \in \mathcal{P}^{1}$ (resp., $(z',z'') \in \mathcal{P}^{2}$). 
Set 
\[     \mathcal{P}_{0} = \{ (z,z'') \in S_{0} \times S''_{0} : \text{ $\exists z' \in S'_{0}$ such that $(z,z') \in \mathcal{P}^{1}$ and $(z',z'') \in \mathcal{P}^{2}$}\},     \]
and 
\[     \mathcal{P} = \mathcal{P}_{0} \cup \big((S \setminus S_{0}) \times \{ * \}\big) \cup \big(\{ * \} \times (S'' \setminus S''_{0})\big).     \]
If $(z,z'') \in \mathcal{P}_{0}$, then $d^{n}(z,z'') \leq d^{n}(z,z') + d^{n}(z',z'')$, for any $z' \in S'$. 
By taking $z'$ as in the definition of $\mathcal{P}_{0}$, we get that $d^{n}(z,z'') \leq d_{1} + d_{2} + 2 \epsilon$. 
If $(z,z'') \in \mathcal{P} \setminus \mathcal{P}_{0}$, we have four possibilities: 
\begin{enumerate}[label=(\roman*)]
\item\label{item:i} 
$z = *$ and $(*,z'') \in \mathcal{P}^{2}$, 

\item\label{item:ii} 
$z'' = *$ and $(z,*) \in \mathcal{P}^{1}$,

\item\label{item:iii} 
$z = *$ and there exists $z' \in S'$ such that $(*,z') \in \mathcal{P}^{1}$ and $(z',z'') \in \mathcal{P}^{2}$,
 
\item\label{item:iv} 
$z'' = *$ and there exists $z' \in S'$ such that $(z,z') \in \mathcal{P}^{1}$ and $(z',*) \in \mathcal{P}^{2}$. 
\end{enumerate}
The fact that $\tilde{d}$ (or $d^{n}$) is a metric tells us that, in the four cases, $d^{n}(z,z'') \leq d_{1} + d_{2} + 2 \epsilon$. 
Since the previous identity holds for any $\epsilon > 0$, we have that $d^{n}(z,z'') \leq d_{1} + d_{2}$, for all 
$(z,z'') \in \mathcal{P}$. 
This proves the fact. 
\end{proof}

For each $i \in \NN_{\geq 2}$, and utilizing the usual convention $[z]$ for the cohomology class of a cocycle $z$, we know that 
\[     m_{i}([z^{n^{1},m^{1}}_{p_{1}^{1},p_{2}^{1}}], \dots,  [z^{n^{i},m^{i}}_{p_{1}^{i},p_{2}^{i}}])     \]
lies in the $k$-submodule generated by the cohomology classes of $\omega^{p_{3}}_{p_{1}}.z^{n + 2 - i,m}_{p_{1},p_{2}}$, where $n = n_{1} + \dots + n_{i}$, $p_{1} = p_{1}^{1} \star \dots \star p_{1}^{i}$, 
$p_{2} \in P_{s} \sqcup \{ - \infty \}$, $p_{3} \in P_{s}$ such that $p_{1} \succcurlyeq p_{3} \succ p_{2}$, and $m$ is a positive integer less than or equal to $\mu^{n + 2 - i}_{p_{1},p_{2}}$. 
We shall denote the set of cocycles whose coefficient in the previous expansion is nonzero by $\mathcal{A}(i, z^{n^{1},m^{1}}_{p_{1}^{1},p_{2}^{2}}, \dots,  z^{n^{i},m^{i}}_{p_{1}^{i},p_{2}^{i}})$. 
We even define the previous set for the case that any of the cocycles $z^{n^{j},m^{j}}_{p_{1}^{j},p_{2}^{j}}$, for $j = 1, \dots, i$, is replaced by the singular element $*$, or there are no nonzero coefficients, 
in which case it is considered to be $\{ * \}$. 
We consider the case $i = 1$ as well, and set $\mathcal{A}(1,z) = \{ z \}$. 
The same comments apply to the elements of $\hat{\mathcal{B}}'$ as well. 

Given $i \in \NN$ and a set $X$, let us denote by $X^{[i]}$ the set of maps from $\{ 1, \dots, i \}$ to $X$.
For $N \in \NN \sqcup \{ +\hskip-0.4mm \infty \}$ we define the \emph{pre-$A_{N}$-bottleneck metric} between $H^{\bullet}(\mathcal{D})$ provided with an $A_{N}$-algebra structure $m_{\bullet}$ 
and $H^{\bullet}(\mathcal{D}')$ provided with an $A_{N}$-algebra structure $m'_{\bullet}$ as 
\[    \operatorname{inf} 
      \underset{\text{\begin{scriptsize}
      \begin{tabular}{c}
      $i = 1, \dots, N$
      \\
      $\phi \in (\bigcup\limits_{n \in \ZZ} \mathcal{P}_{n})^{[i]}$\end{tabular}\end{scriptsize}}}{\operatorname{sup}}  
      \hat{d}(\mathcal{A}(i,\phi(1)_{1}, \dots, \phi(i)_{1}),\mathcal{A}'(i,\phi(1)_{2}, \dots, \phi(i)_{2}))),     \]
where the infimum is taken over all infinite tuples $(\mathcal{P}_{n})_{n \in \ZZ} \in \prod_{n \in \ZZ} \operatorname{PM}(\mathcal{B}^{n}, (\mathcal{B}')^{n})$, and we have denoted $\phi(j) = (\phi(j)_{1}, \phi(j)_{2})$, $\phi(j)_{1} \in \mathcal{B}$, $\phi(j)_{2} \in \mathcal{B}'$, for all $j = 1, \dots, i$. 
We shall denote the previous infimum by 
\[     \hat{\delta}_{N}((H^{\bullet}(\mathcal{D}),m_{\bullet}),(H^{\bullet}(\mathcal{D}'),m'_{\bullet})).      \]
Finally we define the $A_{N}$-bottleneck metric between $H^{\bullet}(\mathcal{D})$ and $H^{\bullet}(\mathcal{D}')$ as 
the quotient metric of $\hat{\delta}_{N}$ under the equivalence relation generated by quasi-isomorphisms of minimal $A_{N}$-algebras. 
More precisely, let $\mathcal{A}_{N}^{\mathrm{min}}(A,A')$ be the set of minimal $A_{N}$-algebra structures 
on either the bigraded module $H^{\bullet}(\mathcal{D})$ or $H^{\bullet}(\mathcal{D}')$. 
We consider the equivalence relation given by $X \sim X'$ if and only if $X$ and $X'$ are quasi-isomorphic $A_{N}$-algebras over $k$. 
Then, the $A_{N}$-bottleneck metric is given by 
\begin{equation}
\label{eq:a_nbm}
     \partial_{\mathcal{B},N}(H^{\bullet}(\mathcal{D}),H^{\bullet}(\mathcal{D}')) = \operatorname{inf} \sum_{j=1}^{M} \hat{\delta}_{N}({}^{j}X , {}^{j}X'),     
\end{equation}
where the infimum is taken over all the choices of any finite set of elements ${}^{j}X$ and ${}^{j}X'$ in the set $\mathcal{A}_{N}^{\mathrm{min}}(A,A')$, for $j = 1, \dots, M$, 
such that the $A_{N}$-algebra structure of ${}^{1}X$ is $m_{\bullet}$ on $H^{\bullet}(\mathcal{D})$, 
the $A_{N}$-algebra structure of ${}^{M}X'$ is $m'_{\bullet}$ on $H^{\bullet}(\mathcal{D}')$,
and ${}^{j}X' \sim {}^{j+1}X$, for all $j = 1, \dots, n-1$.  
This metric (sometimes called the \emph{teleportation distance}) is a fairly classical construction and was first considered by C. Himmelberg in \cite{Him68}, Thm. 2 
(see also \cite{BBI01}, Def. 3.1.12, and the subsequent paragraphs for further considerations). 

We suspect that the proof of the following result is standard, 
but we provide it for convenience.  
%%%%%%%
\begin{proposition}
\label{proposition:metric}
Let $N \in \NN \sqcup \{+ \hskip -0.4mm \infty\}$. 
The function \eqref{eq:a_nbm} defined on the set of generalized barcodes of $P$-filtered dg algebras, 
having a filtration which is continuous on the left and essentially discrete, and of locally finite dimensional cohomology, is a metric.
\end{proposition}
%%%%%%%
\begin{proof}
It suffices to prove that the formula defining $\hat{\delta}_{N}((H^{\bullet}(\mathcal{D}),m_{\bullet}),(H^{\bullet}(\mathcal{D}'),m'_{\bullet}))$ 
satisfies the usual triangular inequality. 
Let us consider again $A$, $A'$ and $A''$ be three $P$-filtered dg algebras, such that the filtrations are essentially discrete and continuous on the left. 
We suppose as usual the multiplicity property for the discrete subsets $P_{s}$, $P'_{s}$ and $P''_{s}$. 
We have thus the $A_{N}$-algebras $(H^{\bullet}(\mathcal{D}),m_{\bullet})$, $(H^{\bullet}(\mathcal{D}'),m'_{\bullet})$ 
and $(H^{\bullet}(\mathcal{D}''),m''_{\bullet})$. 
We will prove that 
\begin{multline*}
     \hat{\delta}_{N}((H^{\bullet}(\mathcal{D}),m_{\bullet}),(H^{\bullet}(\mathcal{D}''),m''_{\bullet})) 
     \\
     \leq 
       \hat{\delta}_{N}((H^{\bullet}(\mathcal{D}),m_{\bullet}),(H^{\bullet}(\mathcal{D}'),m'_{\bullet})) 
       + \hat{\delta}_{N}((H^{\bullet}(\mathcal{D}'),m'_{\bullet}),(H^{\bullet}(\mathcal{D}''),m''_{\bullet})).
\end{multline*}  
Set 
\begin{align*}
d_{1} &= \hat{\delta}_{N}((H^{\bullet}(\mathcal{D}),m_{\bullet}),(H^{\bullet}(\mathcal{D}'),m'_{\bullet})),
\\ 
d_{2} &= \hat{\delta}_{N}((H^{\bullet}(\mathcal{D}'),m'_{\bullet}),(H^{\bullet}(\mathcal{D}''),m''_{\bullet})).
\end{align*} 
If they are infinite there is nothing to prove, so we suppose they are finite. 
Take any real number $\epsilon > 0$ and let 
$\mathcal{P}^{1} \in (\mathcal{P}_{n}^{1})_{n \in \ZZ} \in \prod_{n \in \ZZ} \operatorname{PM}(\mathcal{B}^{n}, (\mathcal{B}')^{n})$ and $\mathcal{P}^{2} \in (\mathcal{P}_{n})_{n \in \ZZ} \in \prod_{n \in \ZZ} \operatorname{PM}((\mathcal{B}')^{n}, (\mathcal{B}'')^{n})$ be such that 
\[     \underset{\text{\begin{scriptsize}
      \begin{tabular}{c}
      $i = 1, \dots, N$
      \\
      $\phi \in (\bigcup\limits_{n \in \ZZ} \mathcal{P}_{n}^{1})^{[i]}$\end{tabular}\end{scriptsize}}}{\operatorname{sup}}  
      \hat{d}(\mathcal{A}(i,\phi(1)_{1}, \dots, \phi(i)_{1}),\mathcal{A}'(i,\phi(1)_{2}, \dots, \phi(i)_{2}))) \leq d_{1} + \epsilon,     \]
      and 
\[     \underset{\text{\begin{scriptsize}
      \begin{tabular}{c}
      $i = 1, \dots, N$
      \\
      $\phi \in (\bigcup\limits_{n \in \ZZ} \mathcal{P}_{n}^{2})^{[i]}$\end{tabular}\end{scriptsize}}}{\operatorname{sup}}  
      \hat{d}(\mathcal{A}'(i,\phi(1)_{1}, \dots, \phi(i)_{1}),\mathcal{A}''(i,\phi(1)_{2}, \dots, \phi(i)_{2}))) \leq d_{2} + \epsilon.     \]
Define $\mathcal{P} = (\mathcal{P}_{n})_{n \in \ZZ}$ such that $\mathcal{P}_{n}$ is the composition of $\mathcal{P}^{1}_{n}$ and 
$\mathcal{P}^{2}_{n}$, for all $n \in \ZZ$, where the composition was recalled in the proof of Fact \ref{fact:2}. 
Take any $i = 1, \dots, N$ and any $\phi \in (\cup_{n \in \ZZ} \mathcal{P}_{n})^{[i]}$. 
By making use of the situations \ref{item:i}, \ref{item:ii}, \ref{item:iii}, or \ref{item:iv}, 
together with its complement, considered in the proof of Fact \ref{fact:2}, we define 
$\phi^{l} \in (\cup_{n \in \ZZ} \mathcal{P}^{l}_{n})^{[i]}$, for $l = 1, 2$, as follows. 
\begin{enumerate}[label=(\emph{\alph*})]
\item
\label{item:a} 
If $\phi(j) = (*,z'')$ and $(*,z'') \in \mathcal{P}^{2}$, by definition of $\mathcal{P}$ it exists $z$ 
such that $(z,*) \in \mathcal{P}^{1}$. 
In this case we define $\phi^{1}(j) = (z,*)$ and $\phi^{2}(j) = (*,z'')$.  

\item\label{item:b}
If $\phi(j) = (z,*)$ and $(z,*) \in \mathcal{P}^{1}$, by definition of $\mathcal{P}$ it exists $z''$ 
such that $(*,z'') \in \mathcal{P}^{2}$. 
Set $\phi^{1}(j) = (z,*)$ and $\phi^{2}(j) = (*,z'')$. 

\item\label{item:c} 
If $\phi(j) = (*,z'')$ and there exists $z' \in S'$ such that $(*,z') \in \mathcal{P}^{1}$ and $(z',z'') \in \mathcal{P}^{2}$, 
define $\phi^{1}(j) = (*,z')$ and $\phi^{2}(j) = (z',z'')$.

\item\label{item:d} 
If $\phi(j) = (z,*)$ and there exists $z' \in S'$ such that $(z,z') \in \mathcal{P}^{1}$ and $(z',*) \in \mathcal{P}^{2}$, 
we choose $\phi^{1}(j) = (z,z')$ and $\phi^{2}(j) = (z',*)$. 

\item\label{item:e} 
If $\phi(j)$ is not of the form $(*,z'')$ or $(z,*)$, by definition of $\mathcal{P}$, there exists $z' \in S'$ such that $(z,z') \in \mathcal{P}^{1}$ and $(z',z'') \in \mathcal{P}^{2}$. 
We set in this situation $\phi^{1}(j) = (z,z')$ and $\phi^{2}(j) = (z',z'')$. 
\end{enumerate}
If either situation \ref{item:a} or \ref{item:b} occurs for some $j = 1, \dots, i$, Fact \ref{fact:2} implies that 
\begin{equation}
\label{eq:des}
\begin{split}
     \hat{d}(\mathcal{A}&(i,\phi(1)_{1}, \dots, \phi(i)_{1}),\mathcal{A}''(i,\phi(1)_{2}, \dots, \phi(i)_{2}))) 
\\ 
&\leq \hat{d}(\mathcal{A}(i,\phi^{1}(1)_{1}, \dots, \phi^{1}(i)_{1}),\mathcal{A}'(i,\phi^{1}(1)_{2}, \dots, \phi^{1}(i)_{2}))) 
\\
&\phantom{\leq} + \hat{d}(\mathcal{A}'(i,\phi^{2}(1)_{1}, \dots, \phi^{2}(i)_{1}),\mathcal{A}''(i,\phi^{2}(1)_{2}, \dots, \phi^{2}(i)_{2}))) 
\\
&\leq d_{1} + d_{2} + 2 \epsilon.     
\end{split}
\end{equation}
If the cases \ref{item:a} or \ref{item:b} do not occur for all $j = 1, \dots, i$, we are only left with the last three cases. 
Note that in any of them we have that $\phi^{1}(j)_{2} = \phi^{2}(j)_{1}$, for all $j = 1, \dots, i$, 
so Fact \ref{fact:2} tells us again that \eqref{eq:des} holds. 
This implies that 
\[     \hat{\delta}_{N}((H^{\bullet}(\mathcal{D}),m_{\bullet}),(H^{\bullet}(\mathcal{D}''),m''_{\bullet})) \leq d_{1} + d_{2} + 2 \epsilon     \]
for all $\epsilon > 0$, which in turn implies the triangular inequality. 
The proposition is thus proved. 
\end{proof}

%%%%%%%
\begin{remark}
\label{remark:a_nbm}
Note that, for $N=1$, the metric we have defined in \eqref{eq:a_nbm} clearly coincides with the usual bottleneck distance \eqref{eq:bm} between the generalized barcodes associated 
to the filtered dg modules underlying $A$ and $A'$. 
Furthermore, given any elements $N \leq N'$ in $\NN \sqcup \{ + \hskip -0.4mm \infty \}$, we have the obvious inequality 
\[     \hat{\delta}_{N}((H^{\bullet}(\mathcal{D}),m_{\bullet}),(H^{\bullet}(\mathcal{D}'),m'_{\bullet})) \leq \hat{\delta}_{N'}((H^{\bullet}(\mathcal{D}),m_{\bullet}),(H^{\bullet}(\mathcal{D}'),m'_{\bullet}))     \]
which in turn implies that 
\[     \partial_{\mathcal{B},N}(H^{\bullet}(\mathcal{D}),H^{\bullet}(\mathcal{D}')) \leq \partial_{\mathcal{B},N'}(H^{\bullet}(\mathcal{D}),H^{\bullet}(\mathcal{D}')).     \]
On the other hand, we note that, by very definition, the distance function defined by \eqref{eq:a_nbm} is manifestly independent of the model chosen for the $A_{N}$-algebra structures on 
the cohomology of the Rees dg algebras. 
\end{remark}
%%%%%%%

%%%%%%%%%%%%%%%%%%%%%%%%%%%%%%%%%%%%%%%%%%%%%%%%%%%%%%%%%%%%%%%%%%%%%%%%%%%%%%%%%%%%%%%%%%%%%%%%%%%%%%%%%%%%%%%%%%%%%%%%%%%%%%%%%%%%%%%%%%%%%%%%%%%%%
\bibliographystyle{model1-num-names}
\addcontentsline{toc}{section}{References}

%%%%%%%%%%%%%%%%%%%%%%%%%%%%%%%%%%%%%%%%%%%%%%%%%%%%%%%%%%%%%%%%%%%%%%%%%%%%%%%%%%%%%%%%%%%%%%%%%%%%%%%%%%%%%%%%%%%%%%%%%%%%%%%%%%%%%%%%%%%%%%%%%%%%%

%%%%%%%%%%%

\end{document}